\documentclass{amsart}
\usepackage{amssymb,amsfonts,amsxtra,amsmath,tikz-cd,xcolor,mathrsfs,verbatim,centernot,enumitem,mathtools}
\allowdisplaybreaks
\usepackage{quiver}
\usepackage{nicematrix}
\usepackage[all]{xy}
\usepackage{dsfont}
\usetikzlibrary{decorations.markings}
\tikzset{negated/.style={
        decoration={markings,
            mark= at position 0.5 with {
                \node[transform shape] (tempnode) {$\backslash$};
            }
        },
        postaction={decorate}
    }
}

\newcommand{\ind}{\mathds{1}}

\newtheorem{theorem}{Theorem}[section]
\newtheorem{lemma}[theorem]{Lemma}
\newtheorem{corollary}[theorem]{Corollary}
\newtheorem{proposition}[theorem]{Proposition}

\newtheorem{question}[theorem]{Question}
\newtheorem{conj}[theorem]{Conjecture}
\newtheorem{notation}[theorem]{Notation}

\newtheorem{definition}[theorem]{Definition}
\newtheorem{remark}[theorem]{Remark}
\newtheorem{example}[theorem]{Example}

\begin{document}

\title{logarithmic versions of Ginzburg's sharp operation for free divisors}

\author{Xia Liao}
\email{xliao@hqu.edu.cn}
\address{Department of Mathematical Sciences, 
Huaqiao University, 
Chenghua North Road 269, 
Quanzhou, Fujian, China}

\author{Xiping Zhang}
\email{xzhmath@gmail.com}
\address{
School of Mathematical Sciences, Key Laboratory of Intelligent Computing and Applications (Ministry of Education),
Tongji University, 
1239 Siping Road, 
Shanghai, China
}

\begin{abstract}
Let $M$ be a complex manifold,  $D\subset M$  a free divisor and $U=M\setminus D$ its complement.   In this paper we study the characteristic cycle $\textup{CC}(\gamma\cdot \ind_U)$ of the restriction of a constructible function $\gamma$ on $U$.  
We globalise Ginzburg's local sharp construction and introduce the log transversality condition, which is a new transversality condition about the relative position of $\gamma$ and $D$. 
We prove that the  log transversality condition is satisfied if either $D$ is normal crossing and $\gamma$ is arbitrary, or $D$ is holonomic strongly Euler homogheneous and $\gamma$ is non-characteristic. Under the log transversality assumption we establish a logarithmic pullback formula for $\textup{CC}(\gamma\cdot \ind_U)$. Mixing Ginzburg's sharp construction with the logarithmic pullback, we obtain a double restriction formula for the Chern-Schwartz-MacPherson class $c_*(\gamma\cdot \ind_{D\cup V})$ where $V$ is any reduced hypersurface in $M$. Applications of our results include the non-negativity of Euler characteristics of effective constructible functions, and CSM classes of hypersurfaces in the open manifold $\mathbb{P}^n\setminus D$ when $D$ is a linear free divisor or a free hyperplane arrangement.
\end{abstract}

\maketitle  

\section{Introduction}
Let $M$ be a complex manifold of dimension $n$ and let $D$ be a reduced divisor on $M$. Given a constructible function $\gamma$ on $M$, in applications involving characteritic classes or D-modules on noncompact manifolds, it is quite useful to know the characteristic cycle of the constructible function $\gamma \cdot \ind_U$ where $U=M\setminus D$. Let $\textup{CC}(\gamma)=\sum n_i [\Lambda_i]$ where each $\Lambda_i$ is an irreducible conic Lagrangian subvariety of $T^*M$. 
In \cite{MR833194} Ginzburg  introduced a ``local sharp operation"  for irreducible conic Lagrangian subvarieties of $T^*M$. In the local setting where  $D=f^{-1}(0)$ is the zero locus of a  function $f\colon M\to \mathbb{C}$ and $\Lambda$ is such a Lagrangian variety, then $\overline{\Lambda^\sharp}\subset T^*M\times \mathbb{C}$ is an $(n+1)$-dimensional irreducible variety over $\mathbb{C}$ regarded as a one parameter family of cycles. Using the techniques of D-module theory  Ginzburg showed that
\begin{equation}
\label{G1local}
\textup{CC}(\gamma \cdot \ind_U) = \lim_{s\to 0}\sum n_i[\overline{\Lambda_i^\sharp}].
\end{equation}
where $s\in \mathbb{C}$ is the parameter for the family and $\lim_{s\to 0}$ is the specialisation of cycles.

When $D$ has irreducible components, a multi-parameter version of the sharp operation is sketched in \cite{MR1021512} Appendix A  and in \cite{MR1769729} the multi-parameter sharp operation was used to compute the Euler characteristics of bounded constructible complexes on semiabelian varieties. A  proof of the multi-parameter  specialization using logarithmic D-modules is given in \cite{WZ21}.
Though the sharp operation is an important construction on its own, it does not see much use in actual computations concerning the restriction of characteristic cycles to complements of divisors since the operation is exotic and is hard to describe in explicit terms (i.e., either by defining equations in the ambient space or by familiar operations such as blow-ups). A majority of later development concerning characteristic cycles on complements of divisors only consider the special case $\gamma=\ind_M$ where $\textup{CC}(\ind_U)$ can be constructed explicitly by blowing up the singular subscheme of $D$, see for example \cite{MR1697199}\cite{MR1795550}.

Evidences suggest that $\textup{CC}(\gamma \cdot \ind_U)$ (hence the sharp operation) is related to the logarithmic (co)tangent bundle. For example when $\gamma=\ind_M$ and $D$ is simple normal crossing (SNC), a classical result of Aluffi \cite{MR1717120} states that  
\begin{equation}\label{chernidentity}
c_*(\ind_U) = c(TM(-\log D)) \cap [M] \/.
\end{equation}
Here $c_*(\ind_U)$ is MacPherson's Chern class transformation of $\ind_U$ which can be computed from $\textup{CC}(\ind_U)$ by taking the dual of its shadow \cite{MR2097164}. For general $\gamma$, relating $\textup{CC}(\gamma \cdot \ind_U)$ to the logarithmic bundle was puzzling for a long time, and only until very recently the relation is clarified by the paper \cite{MRWW} for simple normal crossing $D$. To state this result, recall that there is a sheaf inclusion $\textup{Der}_M(-\log D) \to \textup{Der}_M$ for any  effective divisor $D$. When $D$ is normal crossing, $\textup{Der}_M(-\log D)$ is locally free, so that the sheaf inclusion induces the logarithmic map $j: T^*M \to T^*M(\log D)$ between the cotangent bundle and the logarithmic cotangent bundle. Let $\gamma$ be a constructible function on $M$. We decompose  
\begin{equation*}\label{decom1}
\textup{CC}(\gamma)=\sum_\alpha n_\alpha[T^*_{Z_\alpha}M] + \sum_\beta m_\beta [T^*_{Z_\beta}M]
\end{equation*}
such that each $Z_\alpha \not\subset D$ and each $Z_\beta \subset D$. In \cite{MRWW} the authors globalise the multi-parameter sharp construction for  SNC divisors  and use the  multi-parameter  specialisation in \cite{WZ21} to express $\textup{CC}(\gamma\cdot \ind_U)$ using logarithmic completions:
\begin{theorem}[\cite{MRWW}]\label{t1}
Let $M$ be a smooth complex algebraic variety of dimension $n$ and let $D$ be a   SNC divisor, then 
 $|j^{-1}(\overline{j(T^*_{Z_\alpha}M)}|$ is a union of conormal spaces for each $Z_\alpha$ and 
\begin{equation}\label{resformulacycle}
\textup{CC}(\gamma\cdot \ind_U) = \sum_\alpha n_\alpha [j^{-1}\overline{j(T^*_{Z_\alpha}M)}].
\end{equation}
\end{theorem}
The closures $\overline{j(T^*_{Z_\alpha}M)}$ are analytic subspaces of $T^*M(\log D)$, referred to as the logarithmic conormal spaces of $Z_\alpha$. 
In particular this theorem implies that the Chern class $c_*(\gamma\cdot \ind_U)$ can be computed using these logarithmic conormal spaces.
\begin{corollary}[\cite{MRWW}]\label{c1}
Under the assumption of Theorem \ref{t1},
\begin{equation}\label{rescsmformula}
c_*(\gamma\cdot \ind_U) = (-1)^{n}c\Big(TM(-\log D)\Big)\cap \sum_\alpha n_\alpha p_*\Big(c(\xi^{\vee})^{-1} \cap [\mathbf{P}(\overline{j(T^*_{Z_\alpha}M)}\oplus 1)]\Big)
\end{equation}
where   $p:\mathbf{P}(T^*M(\log D)\oplus 1) \to M$ is the natural projection and $\xi$ is the tautological line subbundle. 
When $\gamma=(-1)^n\cdot \ind_M$ this formula reduces to Equation $(\ref{chernidentity})$.
\end{corollary}

In practice, Theorem \ref{t1} is much easier to use than Ginzburg's original construction. In \cite{MRWW} it is the main input to prove the Huh-Stumfels involution conjecture and in \cite{FMS} it is the main input to relate the toric polar map to the Chern-Schwartz-MacPherson class of a very affine hypersurface. Unfortunately, Theorem~\ref{t1} fails when $D$ is not SNC:
.

\begin{example}\label{counterexample}
Let $M=\mathbb{C}^2$, $D=V(y^2-x^3)$ and let $Z$ be the plane curve parametrised by $x=t^a, y=t^b$ where $a$ and $b$ are positive integers. Set $\gamma=\textup{Eu}^\vee_Z$ so that $\textup{CC}(\gamma)=[T^*_ZM]$. Then Equation  
$(\ref{resformulacycle})$ fails if $1<\frac{b}{a}<2$ and $\frac{b}{a}\neq \frac{3}{2}$. 
\end{example}

Then it is  natural to ask  the following question: 
\begin{question}
Let  $L$ be a line bundle on a complex manifold $M$ and let $f\in H^0(M,L)$ be a nontrivial global section. Assume that the divisor $D=V(f)$ is free, i.e., the sheaf $\textup{Der}_M(-\log D)$ is locally free. Given a constructible function $\gamma$ on $M$, 
when  does Equation $(\ref{resformulacycle})$ hold?
\end{question}

In this paper we partially answer this question in the complex analytic category. 
First we discuss a globalisation of Ginzburg's sharp construction. 
We canonically construct   $(n+1)$-dimensional analytic subvarieties $\overline{J\Lambda^\sharp}$ and 
$\overline{\Lambda^\sharp}$ of the bundle of principal parts $P_M^1L$ and its $L^\vee$-twist $P_M^1L\otimes L^\vee$ (see Definition~\ref{defn; globaljetsharp} and Definition~\ref{defn; globalsharp}). 
They are called the global (jet) sharp constructions since both spaces agree  with Ginzburg's sharp construction in every local chart that trivialises $L$. 
Under the bundle isomorphisms in Corollary~\ref{cor; isoofbundles}  one can also realise the global (jet) sharp constructions  as canonical logarithmic completions  (Proposition~\ref{prop; sharp=logcompletion} and Proposition~\ref{prop; Jsharp=logcompletion}). In particular 
in Proposition~\ref{global-Ginzburg} we reformulate Gizburg's specialisation formula $(\ref{G1local})$ globally.

Based on the global sharp construction we discuss when Equation $(\ref{resformulacycle})$ holds. 
Motivated by the proved case in \cite{MR3830794} where $\gamma=\ind_M$ and $D$ is a free divisor of linear Jacobian type, we introduce the notion of (weak) log transversality in Definition~\ref{defn; logtransverse}. 
The (weak) log transversality condition guaranntees the good behavior of the logarithmic characteristic cycles under (the homology pullback $j^*$) the schematic pullback $j^{-1}$. At least in principle these transversality conditions can be checked by local equations in the algebraic case and be handled by computer softwares. As a part of the built-in mechanism, the Chern class formula \eqref{rescsmformula} is a formal consequence of the transversality condition. More precisely:  
\begin{enumerate} 
\item[$\bullet$] If $\gamma$ is weakly log transverse to $D$, then equation 
\eqref{relation; CCnonequivariant} (see \S \ref{sec:formal consequence} below) holds as homology classes and we have  equation~\eqref{rescsmformula} in  dimension $0$.
\item[$\bullet$] If $\gamma$ is log transverse to $D$, then  equation $(\ref{resformulacycle})$ holds as cycles and  we have equation~\eqref{rescsmformula} in all dimensions.
\end{enumerate}

In \S\ref{geo-test} we study the (weak) log transversality condition in the following cases.
\begin{enumerate}[label=(\roman*)]
\item $D$ is a free divisor of linear Jacobian type $($Definition~\ref{linearJacobianisEu}$)$ and $\gamma= \ind_M$.
\item $D$ is a normal crossing divisor and $\gamma$ is arbitrary.
\item $D$ is a strongly Euler homogeneous free divisor $($Definition~\ref{defn; EH}$)$ and $\gamma=  \ind_M$. 
\item $D$ is a strongly Euler homogeneous free divisor, and the logarithmic stratification of $D$ is non-characteristic  with respect to  $CC(\gamma)$. $($Cf. Definition~\ref{defn; lognonchar}$)$.  
\item $D$ is a  strongly Euler homogeneous free divisor, $M$ is a compact homogeneous space and  $\gamma$ is a `general' constructible function. (Cf. Corollary~\ref{coro; generaltranslate})
\end{enumerate}

\begin{theorem} 
\label{t2}
The constructible function $\gamma$ is  log transverse to $D$ in cases (i) and (ii), and is weakly log transverse to $D$ in cases (iii)-(v). 

If moreover $D$ is holonomic   then $\gamma$ is  log transverse to $D$ in cases (iii)-(v).
\end{theorem}

The divisors in all cases are strongly Euler homogeneous and this assumption is essential (see Proposition~\ref{prop; whyEH}). In particular,  $M$ is weakly log transverse to $D$  if and only if $D$ is strongly Euler homogeneous  (Corollary~\ref{coro; equiSEH}). 
Since the weak log transversality is independent of the choice of the local equation of $D$, this gives a coordinate free characterisation of strong Euler homogeneity. In a recent preprint \cite{Rodriguez25a}, Rodr\'{i}guez gives a criteron for strong Euler homogeneity (for not necessarily free divisors) using matrix rank loci, and it is easy to show that his criterion \cite[Proposition 2.6]{Rodriguez25a} essentially concides with ours in the free case.

We note that in the list above only those reults concerning (iii), (iv) and (v) are new, and among them only (iv) is essential. Since any submanifold of $M$ is  non-characteristic with respect to the zero section $T_M^*M$, (iii) is a special case of (iv). 
Case (i) was first studied in \cite{MR3830794} by the first author, and by \cite[Proposition 1.9]{MR3366865} it is also  a special case of (iii).  Case (v) is a special case of (iv) because of Kleiman's transversality theorem \cite{Kleiman74}.

The major ingredient for proving the log transversality of $\gamma$ and $D$ in (ii) was contained in \cite{WZ21}, stated in the language of log D-modules. Reference \cite{MRWW} exploited the homological consequence of this result, leading to Theorem \ref{t1}. In hindsight, we realise that the log D-module machinary is not necessary for establishing the log transversality of $\gamma$ and $D$ in (ii). Instead, a much ealier paper \cite[\S 1]{BMM00} did the job and proved the nessesary result using only basic tools from complex analytic geometry. We will explain this view in Theorem~\ref{theo; NCimpliesLT}. On the other hand, we prove that weak log transversalities can be tested by analytic paths (Lemma~\ref{lemma:curvetest}) which gives a simple proof of the weak log transversality of $\gamma$ and $D$ in (ii). In a followup paper \cite{LZ24-2}  we will  discuss  more general algebraic and geometric criteria for the weak log transversality condition. 
We also conjecture that the weak log transversality condition characterises normal crossing divisors.

\begin{conj}[Conjecture~\ref{conj; SNC}]
Let $D$ be a reduced free divisor of $M$. Then $D$ is normal crossing if and only if any constructible function $\gamma$ is weakly log transverse to $D$.
\end{conj}

An immediate consequence of (iii) is the following statement about Euler characteristic.

\begin{corollary}[Corollary~\ref{coro; SEHeulercharacteristic}] 
Let  $D$ be a  strongly Euler homogeneous free divisor in a complex compact manifold $M$ and $U=M\setminus D$ be its complement. Then we have
\begin{equation}\label{SEH;eulercharacteristic}
\chi(U)= \int_M c(\textup{Der}_M(-\log D))\cap [M] \/.
\end{equation}

\end{corollary}

When $D$ is free, holonomic and strongly Euler homogeneous, $\ind_M$ is log transver to $D$ by Theorem \ref{t2} case (iii). Therefore equation \eqref{chernidentity} holds. Since for holonomic free divisors, strong Euler homogeneity is equivalent to linear Jacobian type \cite[Theorem 4.7]{MR3366865}, this gives another proof that equation \eqref{chernidentity} holds for free divisors of linear Jacobian type \cite{MR3830794}. 
 
Corollary \ref{coro; SEHeulercharacteristic} also touches an interesting point about the relation among the Chern class identity  \eqref{chernidentity}, strongly Euler homogeneity (SEH) and the logarithmic comparison theorem (LCT). A divisor $D$ is said to satisfy LCT if the inclusion of the logarithmic de Rham complex $\Omega^\bullet_M(\log D)$ into the meromorphic de Rham complex $\Omega^\bullet_M(\star \textup{D})$  is quasi-isomorphic. By \cite[Theorem 4.7]{MR3366865}, LCT is equivalent to SEH for holonomic free divisors. Therefore we have the following implications for a free divisor $D$.
\[
\text{holonomic LCT} \Leftrightarrow \text{holonomic SEH} \Rightarrow \text{equation \eqref{chernidentity}}.
\]

Dropping down the holonomicity assumption, it is still a conjecture \cite{MR1898392} that for free divisors LCT implies SEH. This conjecture was proved in \cite{MR1898392} and \cite{MR2231201} when $\dim M=2,3$. Recently Rodr\'{i}guez proves the conjecture in \cite{Rodriguez25b} when $\dim M=4$. It is proved in \cite{MR3847339} that LCT implies the equallity of \eqref{chernidentity} in degree $0$ for compact $M$ and free $D$.
In summary, we have the following implications and conjectures for free divisors. Here by WLT we mean $M$ is weakly log transverse to $D$. 
\[
\text{LCT} \xRightarrow{?}    \text{SEH} \Leftrightarrow  \text{WLT}, \quad \text{SEH} \centernot\implies \text{LCT}.
\]
If $M$ is furthermore compact, then
\[
\text{LCT} \Rightarrow  \text{equality \eqref{SEH;eulercharacteristic} holds} \Leftarrow \text{SEH}.
\]

In \S \ref{sec; doubleres} we study the restriction of  $\ind_M$  to  the complement of the union of a free divisor $D$  and a reduced hypersurface $V$.  We introduce a logarithmic  variant of  Ginzburg's sharp construction (Definition~\ref{defn:JsharplogD}) and prove the following formula in Theorem~\ref{theo; csmdoubleres}.
\begin{theorem}[Double Restriction Formula]
\label{t3}
Assume  $\ind_{V}$ is log transverse to the holonomic strongly Euler homogeneous free divisor $D$.
Then we have  
\[c_{*}(\ind_{M\setminus(V\cup D)}) = c(TM(-\log D))c(\mathscr{O}_M(V))^{-1}\cap \Big([M]-s\left(\mathcal{J}_V(\log D), M\right)^\vee\otimes_M \mathscr{O}_M(V) \Big).
\]
Here  $\mathcal{J}_V(\log D)$ is the logarithmic Jacobian ideal sheaf defined in Definition~\ref{defn; J_VlogD}.
\end{theorem}

In \S \ref{sec; nonnegative} we discuss a few results about the non-negativity of $\chi(M\setminus D,\gamma)$, where $\gamma$ is a constructible function on $M$ with effective characteristic cycle. The basic observation is that when $\textup{CC}(\gamma)$ is effective, its log completion $\textup{CC}(\gamma)^{\log}$ is also effective. Hence if $T^*M(\log D)$ is nef and $\gamma$ is log transverse to $D$, we have $\chi(M\setminus D,\gamma)\geq 0$. We also discuss the consequences of effective $\textup{JCC}(\gamma)$.

In \S \ref{generalisation-FMS} we consider the special case $M=\mathbb{P}^n$. We apply the double restriction formula to generalise a result \cite[Theorem 1.2]{FMS} of Fassarella, Medeiros and Salom\~{a}o to holonomic linear free divisors 
(Corollary~\ref{coro; linearfreecsm}). We also give a numerical criterion (Corollary~\ref{coro; linearfree}) for the non-LCT of holonomic linear free divisors. 
The formula \cite[Theorem 1.2]{FMS} concerns the multi-degrees of the toric polar map  while a similar formula of Aluffi  \cite[Theorem 2.1]{Aluffi03} concerns the multi-degrees  of gradient map. In an earlier draft of \cite{FMS} the authors explained the  equivalence of the  two formulas  under the normal crossing intersection hypothesis. 
In  Theorem~\ref{theo; Segreimpliesintegral} we  provide a general explanation of  the equivalence  by showing that different ideal sheaves given by the gradient map  and the toric polar map have  the same integral closure, and therefore produce  the  same Segre class.
The case of centrally free projective hyperplane arrangement  is considered  in \S\ref{sec; hyperplanes}, where we discuss a variant of a recent result \cite[Theorem 1.1]{RW24} of Reinke and Wang.

\vspace{0.5cm}
\noindent
{\bf Acknowledgements}: 
The second author is grateful for the helpful comments and discussions for \S\ref{sec; nonnegative} from  Jie Liu, Yongqiang Liu,  Botong Wang, Dingxin Zhang and Ziwen Zhu.

\section{Preliminaries}
\label{sec; pre}
\subsection{Chern class transformations}
\label{sec; chern}
Let $X$ be a complex algebraic variety. A constructible function on $X$ is a function $\gamma: X \to \mathbb{Z}$ such that $\gamma^{-1}(i)$ is a constructible subset of $X$ for any $i$. In his seminal paper \cite{MR361141}, MacPherson associated a homology class $c_*(\gamma)$ to any constructible function $\gamma$, and this association is functorial with respect to proper maps between algebraic  varieties. More precisely, there is a commutative square for any proper map $X\to Y$ between algebraic  varieties

\begin{equation*}
\begin{tikzcd}
CF(X) \arrow[d] \arrow[r,"c_*"] & H_*(X) \arrow[d]\\ 
CF(Y) \arrow[r,"c_*"] & H_*(Y)
\end{tikzcd}
\end{equation*}
where $CF(\bullet)$ is the group of constructible functions and $H_*(\bullet)$ is the Borel-Moore homology. The transformation $c_*$ also satisfies the normalisation condition $c_*(\ind_X)=c(TX)\cap [X]$ for any nonsingular $X$. Below we present the bare minimum of the Chern class transformation which matters to this paper. Readers are directed to \cite{MR2143077} for proofs and details.

When $X$ is embedded as a closed subvariety of a complex manifold $M$ of dimension $n$, the Chern class transformation $c_*(\gamma)$ can be computed by characteristic cycles. First, for any closed subvariety $Z$ of $M$, the conormal space $T^*_ZM$ is the closure in $T^*M$ of the set
\begin{equation*}
\{(z,\alpha)\in T^*M \ | \ z\in Z_{reg}, \ \alpha(T_zZ)=0\}.
\end{equation*}
This is a conical Lagrangian subvariety.
The dual local Euler obstruction $\textup{Eu}_Z^\vee$ is a constructible function supported on $Z$ whose value at $p\in Z$ is the local intersection number $\#_z(dr(B_\epsilon)\cap T^*_ZM)$ where $B_\epsilon$ is a small Euclidean ball around $z$ and $r$ is the squared Euclidean distance function to $z$. Dual local Euler obstructions form a basis for the free abelian group $CF(M)$. Indeed, $\textup{Eu}^\vee_Z(p)=(-1)^{\dim Z}$ at a point $p\in Z_{reg}$. Therefore, one may find closed subvarieties $Z_\alpha$ and integers $n_\alpha$ for some finite index set $I$ such that 
\begin{equation*}
\gamma = \sum_{\alpha\in I} n_\alpha \textup{Eu}_{Z_\alpha}^\vee.
\end{equation*}
This expansion of $\gamma$ determines its characteristic cycle.
\begin{definition}\label{CC-def}
The characteristic cycle $\textup{CC}(\gamma)$ of the constructible function $\gamma$ is 
\begin{equation*}
\textup{CC}(\gamma):=\sum_{\alpha\in I} n_\alpha [T^*_{Z_\alpha}M].
\end{equation*}
\end{definition} 

Let $\eta$ be the tautological line subbundle of $\mathbf{P}(T^*M\oplus 1)$ and let $q:\mathbf{P}(T^*M\oplus 1)\to M$ be the natural projection. The Chern class transformation $c_*(\gamma)$ can be computed by the following formula
\begin{equation*}
c_*(\gamma)= (-1)^{n}c(TM)\cap \sum_{i\in I}n_i q_*\Big(c(\eta^\vee)^{-1}\cap [\mathbf{P}(T^*_{Z_i}M\oplus 1)]\Big).
\end{equation*}

One can give a convenient reformulation of the Chern class transformation via its dual $\breve{c}_*(\gamma)$. By definition, $\breve{c}_*(\gamma)$ is given by
\begin{equation*}
\breve{c}_i(\gamma)=(-1)^ic_i(\gamma)
\end{equation*}
where $\breve{c}_i(\gamma)$ (respectively $c_i(\gamma)$) are the $i$-dimensional component of $\breve{c}_*(\gamma)$ (respectively $c_*(\gamma)$). One has the formula (see \cite{MR2097164} or \cite{MR2143077})
\begin{equation*}
\breve{c}_*(\gamma)= c(T^*M)\cap \sum_{i\in I}n_i\cdot s(T^*_{Z_i}M)
\end{equation*}
where   
\begin{equation*}
s(T^*_{Z_i}M)=q_*\Big(c(\eta)^{-1}\cap [\mathbf{P}(T^*_{Z_i}M\oplus 1)]\Big)
\end{equation*}
 is the Segre class of the cone $T^*_{Z_i}M$.

By \cite{MR1483328}, MacPherson's graph construction can also be carried out in the complex analytic category. This implies that the theory of Chern class transformation can be extended to the complex analytic category. For each complex analytic variety $X$ and any analytically constructible function $\gamma$, one can associate a characteristic cycle $\textup{CC}(\gamma)$ and a homology class $c_*(\gamma)$ by the same formulas as presented above, except that  the family $Z_\alpha \ (\alpha\in I)$ is now locally finite in the analytic context.

\subsection{Free divisors}\label{freedivisor}
Let $M$ be a complex manifold of dimension $n$ and let $L$ be a line bundle on $M$. Let $h
\in H^0(M,L)$ be a nontrivial global section and  $D=V(h)$ be an effective reduced divisor. For any local representative of $h$ on an open subset $U\subset M$, one can define the module of derivations preserving $h$ as follows.
\begin{equation*}
\textup{Der}_U(-\log h)=\{\delta \in \textup{Der}_U \ | \  \delta(h) \in \mathscr{O}_U\cdot h\}.
\end{equation*}
These locally defined modules of derivations glue together. The result is the sheaf of logarithmic derivations along $D$, denoted by $\textup{Der}_M(-\log D)$. Saito introduced  $\textup{Der}_M(-\log D)$ in \cite{MR586450} and proved it is coherent and reflexive. The dual of $\textup{Der}_M(-\log D)$ is the sheaf of logarithmic 1-forms and is denoted by $\Omega^1_M(\log D)$.

\begin{definition}[\cite{MR586450}]
We say $D$ is free at $p\in D$ if $\textup{Der}_{M,p}(-\log D)$ is a free $\mathscr{O}_{M,p}$-module. We say $D$ is a free divisor if it is free at every $p\in D$. 
\end{definition}

Given a free divisor $D$, the logarithmic tangent bundle $TM(-\log D)$ is the vector bundle whose sheaf of sections is $\textup{Der}_M(-\log D)$. The logarithmic cotangent bundle $T^*M(\log D)$ is the dual of $TM(-\log D)$, whose sheaf of sections is $\Omega^1_M(\log D)$. The natural inclusion $\textup{Der}_M(-\log D) \hookrightarrow \textup{Der}_M$ induces a linear map between rank $n$ vector bundles $j:T^*M \to T^*M(\log D)$. We  call it the logarithmic map of the pair $(M,D)$.  

Next we recall the definition of a divisor being  of linear Jacobian type, being strongly Euler homogeneous and  being holonomic. These definitions make sense without assuming $D$ is a free divisor, but these properties are often studied in the context of free divisors.

\begin{definition}[\cite{MR3366865}]
\label{definition:Jacobianlineartype}
Let $\mathcal{J}_D$ be ideal sheaf of $\mathscr{O}_M$ that defines the singular subscheme $\textup{ Sing}(D)$ of $D$. Locally $\mathcal{J}_D$  is represented by the ideal $(h,\partial_{x_1}(h),\ldots,\partial_{x_n}(h))$ for a local equation $h$ of $D$ and a local coordinate system $(x_1, \cdots ,x_n)$ of $M$. The divisor $D$ is of linear Jacobian type if the natural morphism
$\textup{Sym}^\bullet_{\mathscr{O}_M}(\mathcal{J}_D)\rightarrow \textup{Rees}^\bullet_{\mathscr{O}_M}(\mathcal{J}_D)$ is an isomorphism.
\end{definition}

\begin{definition}[\cite{MR3366865}]
\label{defn; EH}
The divisor $D$ is  Euler homogeneous if for any $p\in D$, there exists a  local equation $h\in \mathscr{O}_{M,p}$ of $D$  and  a local vector field $\chi\in Der_{M,p}$ such that $\chi(h)=h$. Such a vector field   is called a  Euler vector field of $h$ at $p$. 

The divisor $D$ is strongly Euler homogeneous if for any $p\in D$, there exists a local equation $h\in \mathscr{O}_{M,p}$ of $D$ and a local Euler vector field $\chi\in \mathfrak{m}_{M,p}\cdot Der_{M,p}$  such that $\chi(h)=h$.  Such a vector field   is called a strongly Euler vector field of $h$ at $p$.
\end{definition}

Similar to quasihomogeneity of divisors, the notion of strongly Euler homogeneity involves a choice of local defining equations. The local strongly Euler vector field $\chi$ only exists for a carefully chosen local equation of the divisor. Such a local equation is said to be strongly Euler homogeneous.

\begin{proposition}[{\cite[Proposition 1.9]{MR3366865}}]\label{linearJacobianisEu}
If $D$ is of linear Jacobian type, then $D$ is strongly Euler homogeneous.
\end{proposition}

At any point $p$, the evaluations at $p$ of the vector field germs $\delta\in \textup{Der}_{M,p}(-\log D)$ form a  linear subspace of $T_pM$, which we denote by $\textup{Der}_M(-\log D)(p)$. There exists a unique stratification $\{D_\alpha,\alpha\in I\}$ of $M$ with the properties that
\begin{enumerate}[label=(\roman*)]
\item Stratum $D_\alpha$ is a smooth connected immersed submanifold of $M$, and $M$ is the disjoint union of the strata.
\item If $p\in D_\alpha$, then $T_pD_\alpha=\textup{Der}_M(-\log D)(p)$.
\end{enumerate}
The stratification $\{D_\alpha,\alpha\in I\}$ is called the logarithmic stratification of $M$. A stratum is called a logarithmic stratum. By \cite[(3.6)]{MR586450} each $D_\alpha$ is a locally closed connected analytic submanifold of $M$.  

\begin{definition}[\cite{MR586450}]
\label{defn; holonomic}
$D$ is called a holonomic divisor if the logarithmic stratification $\{D_\alpha,\alpha\in I\}$ is locally finite.
\end{definition}


\begin{definition}
\label{defn; lognonchar}
Let $CC(\gamma)=\sum_\beta n_\beta [T^*_{Z_\beta}M]$ be a characteristic cycle. We say a logarithmic stratum $D_\alpha$   is non-characteristic with respect to  $CC(\gamma)$, if 
$T^*_{Z_\beta}M\cap T_{D_\alpha}^*M\subset T_M^*M$ 
for any $Z_\beta\not\subset D_\alpha$. 
We say the logarithmic stratification is non-characteristic with respect to  $CC(\gamma)$ if every logarithmic stratum   is non-characteristic with respect to  $CC(\gamma)$. 
\end{definition}

\subsection{The Multi-degrees of  Rational Morphisms between Projective Spaces}
\label{sec; multidegree}
Given homogeneous polynomials $f_0, \ldots, f_n \in \mathbb{C}[x_0, \cdots ,x_n]$ of the same degree, we define a rational morphism 
\[
\varphi\colon \mathbb{P}^n \dashrightarrow \mathbb{P}^n \/, x\mapsto [f_0(x): f_1(x): \cdots :f_n(x)]  \/. 
\]
The  multi-degrees of $\varphi$ is defined as follows. 
We consider the image closure of $\varphi$:
\[
\Gamma_\varphi:= \overline{
\left\lbrace
\left(x, [f_0(x): f_1(x): \cdots :f_n(x)]\right)\big|  x\notin V(f_0, f_1,\cdots ,f_n) 
\right\rbrace
}\subset \mathbb{P}^n\times \mathbb{P}^n
\]
This is nothing but the blowup of $\mathbb{P}^n$ along the ideal $\mathcal{I}_\varphi=(f_0, \cdots ,f_n)$. 
Let $\pi_1$ and $\pi_2$ be the first and second projections to $\mathbb{P}^n$ from $\mathbb{P}^n\times \mathbb{P}^n$ 
and $H_i$ be the pull-back  hyperplane classes from $\pi_i$ for $i=1,2$. 
The fundamental class $[\Gamma_\varphi]$ is then written as
\[
[\Gamma_\varphi]=\Big(\sum_{j=0}^n d_j H_1^j H_2^{n-j}\Big) \cap [\mathbb{P}^n\times \mathbb{P}^n].
\]  

\begin{definition}
\label{defn; multidegree}
We call the integer $d_k$  the $k$-th multi-degree of the rational morphism $\varphi$.
\end{definition}

The last multi-degree $d_n$ is the usual degree of $\phi$.

\begin{remark}
we may also view $\mathcal{I}_\varphi$ as an ideal of $\mathscr{O}_{\mathbb{C}^{n+1}}$.  One can show that 
\[
d_k=\textup{Muti}_0 \left(\pi_1\left(\tilde{H}_1\cap \cdots \cap \tilde{H}_k\cap Bl_{\mathcal{I}_\varphi}\mathbb{C}^{n+1}\right)\right)
\]
where $\tilde{H}_1,\ldots, \tilde{H}_k$ are pullbacks of general hyperplanes in $\mathbb{P}^n$ along $\pi_2: Bl_{\mathcal{I}_\varphi}\mathbb{C}^{n+1} \to \mathbb{P}^n$. This number is also called the $k$-th relative polar multiplicity $m_k(\mathcal{I}_\varphi, \mathbb{C}^{n+1})$ in \cite[\S 2.1]{G-G99}.   
\end{remark}
\begin{notation}
Let $\mathcal{I}$ be an ideal sheaf on $\mathbb{P}^{n}$. We use the notation $s(\mathcal{I}, \mathbb{P}^n)$ for the Segre class of the subscheme defined by $\mathcal{I}$ in $\mathbb{P}^n$. 
\end{notation}
The following relation between multi-degrees and characteristic classes is observed by Aluffi in \cite[Proposition 3.1]{Aluffi03}. 
\begin{proposition}
\label{prop; muldegandsegre}
If the common degree of  $f_0,\ldots, f_n$ is $r$, then
\[
\sum_{i=0}^n (-1)^id_i [\mathbb{P}^{n-i}]=c(\mathscr{O}(r))^{-1}\cap \left([\mathbb{P}^n]-s(\mathcal{I}_{\varphi}, \mathbb{P}^n)^\vee\otimes  \mathscr{O}(r) \right) \/.
\] 
\end{proposition}

The reader should refer to \cite{Aluffi03} for the precise meaning of the the dual and tensor notation in the formula.
 
\subsection{Centrally Free Divisors in the Projective Space}
\label{sec; centralfree}
Let $f\in \mathbb{C}[x_0,\ldots,x_n]$ be a homogeneous polynomial. It defines a projective hypersurface $D=V(f)\subset \mathbb{P}^n$ and its affine cone $\hat{D}=V(f)\subset \mathbb{C}^{n+1}$.

\begin{definition}
We say $D$ is  centrally free  if its affine cone $\hat{D}$ is free at $0\in \mathbb{C}^{n+1}$.
\end{definition}
The definition implies that $\hat{D}$ is free at a neighbourhood of $0\in \mathbb{C}^{n+1}$, hence is free everywhere because $\hat{D}$ is $\mathbb{C}^*$-invariant. By the Quillen-Suslin theorem $H^0(\mathbb{C}^{n+1}, \textup{Der}_{\mathbb{C}^{n+1}}(-\log \hat{D}))$ is a free module of the form
$
\bigoplus_{i=0}^{n} \mathbb{C}[x_0,\ldots,x_n] \delta_i \/,
$
where $\delta_i$ is a homogeneous derivation of degree $d_i-1$, i.e. $\delta_i=\sum_j\delta_{i,j}\partial_{x_j}$ such that $\delta_{i,j}$ is a homogeneous polynomial of degree $d_i$ (the differential operator $\partial_{x_j}$ has degree $-1$). The degrees $d_0,\ldots, d_n$ are independent of the choice of the basis. We always choose $\delta_0=\sum_i x_i\partial_{x_i}$ the Euler derivation.

\begin{proposition}\label{prop:coneproperties}\hfill
\begin{enumerate}
\item If $D$ is centrally free, then $D$ is a free divisor. 
\item If  $\hat{D}$ is strongly Euler homogeneous, then $D$ is strongly Euler homogeneous.
\item If $\hat{D}$ is holonomic, then $D$ is holonomic.
\end{enumerate} 
\end{proposition}

\begin{proof}
At any $p\in \hat{D}$ different from $0\in \mathbb{C}^{n+1}$, the germ $(\hat{D},p)$ is a product of $D$ with a trivial factor $\mathbb{C}$. Property (1) follows from \cite{MR1363009} Lemma 2.2, and (2) follows from \cite{MR2231201} Lemma 3.2. Property (3) is trivial.
\end{proof}
\begin{proposition}
\label{prop; chernderlog}
For a centrally free divisor $D$, there is a short exact sequence
\[
\begin{tikzcd}
0  \arrow{r} & \mathscr{O}_{\mathbb{P}^n} \arrow{r} & \oplus_{i=0}^n \mathscr{O}_{\mathbb{P}^n}(1-d_i) \arrow{r} & \textup{Der}_{\mathbb{P}^n}(-\log D)\arrow{r} & 0
\end{tikzcd}
\]
In particular, $\textup{Der}_{\mathbb{P}^n}(-\log D)\cong \oplus_{i=1}^n \mathscr{O}_{\mathbb{P}^n}(1-d_i)$.
\end{proposition}
\begin{proof}
Observe that $\mathbb{C}^*$-invariant vector fields on $\mathbb{C}^{n+1}$ descends to vector fields on $\mathbb{P}^n$, and a rational vector field on $\mathbb{C}^{n+1}$ is $\mathbb{C}^*$-invariant if and only if it is of degree $0$. Degree 0 rational vector fields on $\mathbb{C}^{n+1}$ logarithmic along $\hat{D}$ are of the form $\sum_i f_i\delta_i$ where $f_i$ is a rational function of degree $1-d_i$, and this gives the middle term $\oplus_{i=0}^n \mathscr{O}_{\mathbb{P}^n}(1-d_i)$.
\end{proof}

\begin{example}
\label{exam; logcothyparr}
Assume that $k\leq n$.
Let $D=V(l_0\cdots l_k)\subset \mathbb{P}^n$ be the disjoint union of $k+1$ hyperplanes such that $(l_0, \cdots ,l_k)$ are independent linear functions on $\mathbb{C}^{n+1}$. Then $D$ is centrally free and 
\[
c(T^*\mathbb{P}^n(\log D))=(1-H)^{n-k}.
\]
\end{example}

\begin{definition}[\cite{MR2228227}\cite{MR2521436}]
\label{defn; linearfreedivisor}
Let $D$ be a centrally free divisor. We say $\hat{D}$ is linear free if $d_0=\ldots = d_n=1$. 
\end{definition}

A basic example of linear free divisor is the SNC divisor $D=V(x_0\cdots x_n)$ consisting of all coordinate hyperplanes. We deduce from Proposition \ref{prop; chernderlog} that $T\mathbb{P}^n(-\log D)$ is trivial when $\hat{D}$ is linear free. Linear free divisors have many interesting properties. In \cite{MR2521436} the authors showed that global  logarithmic comparison theorem always holds for reductive linear  free divisors. 
When the reductive linear  free divisor is holonomic, it is strongly Euler
homogeneous if and only if  the logarithmic comparison theorem holds \cite[Theorem 1.6]{GS10}.

\section{Ginzburg's sharp operation}
\label{sec; sharp} 
Let $M$ be a complex manifold of dimension $n$ and let $f: M\to \mathbb{C}$ be a non-constant global holomorphic function. Set $D=f^{-1}(0)$, $U=M\setminus D$ and let $\gamma$ be a constructible function on $M$. First, we review Ginzburg's construction of $\textup{CC}(\gamma\cdot \ind_U)$ and relate it to the logarithmic completion of $\textup{CC}(\gamma)$. To accomplish this purpose, it is enough to concentrate on the case $\gamma=\textup{Eu}^\vee_Z$ where $Z\not\subset D$ is a closed subvariety. To simplify notation, let $\Lambda=T^*_ZM$ from now. 

In \cite{MR833194} Ginzburg defined a $(n+1)$-dimensional analytic subspace $\Lambda^\sharp_f \subset T^*U \times \mathbb{C}^*$:
\begin{equation}\label{equation:localginz}
\Lambda^\sharp_f := \Big\{(\xi + s\frac{df}{f}(x),s)\ \vert \ (x,\xi) \in \Lambda, x\in U, s\in\mathbb{C}^* \Big\}.
\end{equation}
Let $\overline{\Lambda^\sharp_f}$ be the closure of $\Lambda^\sharp_f$ in $T^*M \times \mathbb{C}$ and  let $\pi_2 : \overline{\Lambda^\sharp_f} \to \mathbb{C}$ be the natural projection. When there's no confusion, we then use the shorthand notation $\overline{\Lambda^\sharp}$. 

\begin{proposition}[{\cite[Theorem 3.2]{MR833194}}]\label{CC-G}
The characteristic cycle of the restricted constructible function $\textup{Eu}^\vee_Z\cdot \ind_U$ is given by
\begin{equation*}
\textup{CC}(\textup{Eu}^\vee_Z\cdot \ind_U) = [\pi_2^{-1}(0)].
\end{equation*}
\end{proposition}

If we set Cartesian diagram  
$$
\begin{tikzcd}
\pi_2^{-1}(0) \arrow[d] \arrow[r] & \overline{\Lambda^\sharp} \arrow[d,"\pi_2"']\\
0 \arrow[r,"i"] & \mathbb{C}
\end{tikzcd}\/,
$$
the above formula can be reformulated as 
$$
\textup{CC}(\textup{Eu}^\vee_Z\cdot \ind_U)=i^*([\overline{\Lambda^\sharp}])
$$ 
where $i^*$ is the specialisation map in Borel-Moore homology (see \cite{MR1644323} Example 19.2.1, cf. also \S 10.1 in loc.cit.). It is helpful to understand this construction using the following pictures.

\begin{figure}[!h]
\begin{tikzpicture}
\draw (-2,0) -- (2,0) ;
\draw [dashed, -latex] (0,0) -- (0,2) ;

\filldraw [gray] (0,0) circle (1.5pt);
\draw (0,0) node[below]{\footnotesize $D$};
\draw (2,0) node[below]{\footnotesize $M$};
\draw (0,2) node[above]{\footnotesize $dx_1,\ldots,dx_m$};

\draw (-2,0.5) .. controls (-1,0.6) and (-0.3,1.2) ..  (-0.2,2);

\draw (2,0.3) .. controls (1,0.4) and (0.3,1) .. (0.2,2);

\begin{scope}[xshift=5cm]
\draw (-2,0) -- (2,0) ;
\draw [dashed, -latex] (0,0) -- (0,2) ;

\filldraw [gray] (0,0) circle (1.5pt);
\draw (0,0) node[below]{\footnotesize $D$};
\draw (2,0) node[below]{\footnotesize $M$};
\draw (0,2) node[above]{\footnotesize $dx_1,\ldots,dx_m$};

\draw (-2,0.3) .. controls (-1,0.6) and (-0.3,1.2) ..  (-0.2,2)
  \foreach \p in {20,40,60,80} {
   node[pos=\p*0.01]
      (N \p){}
      }
      ;
      
\foreach \p in {20,40,60,80} 
{
\draw [red] (N \p.center) --  ++(0.1,0.2);
\draw [red] (N \p.center) --  ++(-0.1,-0.2);
\draw [red] (N \p.center) --  ++(-0.1,0.2);
\draw [red] (N \p.center) --  ++(0.1,-0.2);
}

\draw (2,0.3) .. controls (1,0.4) and (0.3,1) .. (0.2,2)
  \foreach \p in {20,40,60,80} {
   node[pos=\p*0.01]
      (M \p){}
      }
      ;
      
\foreach \p in {20,40,60,80} 
{
\draw [red] (M \p.center) --  ++(0.2,0.2);
\draw [red] (M \p.center) --  ++(-0.2,-0.2);
\draw [red] (M \p.center) --  ++(-0.2,0.2);
\draw [red] (M \p.center) --  ++(0.2,-0.2);
}

\end{scope}

\begin{scope}[xshift=10cm]
\draw (-2,0) -- (2,0) ;
\draw [dashed, -latex] (0,0) -- (0,2) ;

\filldraw [gray] (0,0) circle (1.5pt);
\draw (0,0) node[below]{\footnotesize $D$};
\draw (2,0) node[below]{\footnotesize $M$};
\draw (0,2) node[above]{\footnotesize $dx_1,\ldots,dx_m$};

\draw (-2,0.25) .. controls (-1,0.3) and (-0.3,0.6) ..  (-0.2,1)
  \foreach \p in {20,40,60,80} {
   node[pos=\p*0.01]
      (R \p){}
      }
      ;
      
\foreach \p in {20,40,60,80} 
{
\draw [red] (R \p.center) --  ++(0.1,0.2);
\draw [red] (R \p.center) --  ++(-0.1,-0.2);
\draw [red] (R \p.center) --  ++(-0.1,0.2);
\draw [red] (R \p.center) --  ++(0.1,-0.2);
}

\draw (2,0.15) .. controls (1,0.2) and (0.3,0.5) .. (0.2,1)
  \foreach \p in {20,40,60,80} {
   node[pos=\p*0.01]
      (S \p){}
      }
      ;
      
\foreach \p in {20,40,60,80} 
{
\draw [red] (S \p.center) --  ++(0.2,0.2);
\draw [red] (S \p.center) --  ++(-0.2,-0.2);
\draw [red] (S \p.center) --  ++(-0.2,0.2);
\draw [red] (S \p.center) --  ++(0.2,-0.2);
}

\end{scope}
\end{tikzpicture}
\caption{Ginzburg's construction of $\textup{CC}(\textup{Eu}^\vee_Z\cdot \ind_U)$}
\end{figure}

\begin{itemize}
\item the image of $Z\cap U$ under the map $\frac{df}{f}: U \to T^*U$,
\item the image of $\Lambda\vert_U$ under the map $+\frac{df}{f}:\Lambda\vert_U \to T^*U$ where the red `cones' are representative conormal vectors to $Z$,
\item the image of $\Lambda\vert_U$ under the map $+\frac{df}{2f}: \Lambda\vert_U \to T^*U$, i.e. $\pi^{-1}_2(\frac{1}{2})$.
\end{itemize}

Since  $\pi_2$ is flat, $\pi^{-1}_2(0) = \lim_{s\to 0} \pi^{-1}_2(s)$. The meaning of Ginzburg's construction of $\textup{CC}(\textup{Eu}^\vee_Z\cdot 1_U)$ is that we have to draw the diverging conormal vectors near the boundary $D$ back towards the zero section. In applications, Proposition \ref{CC-G} is not easy to use partly because the construction `pulling back from infinity' is not immediately rephrased as those standard constructions in algebraic geometry such as blowup or deformation to the normal cone. In our next paper \cite{LZ24-3} we will relate Ginzburg's construction to certain Rees algebras of modules.

Suppose $D$ is free for a moment. Motivated by \cite{MRWW} we consider the following diagram.
\begin{equation}
\label{diagram; sharp}
\begin{tikzcd}
T^*M \arrow[r, hook,"i"] \arrow[rr, bend left, "j"]& T^*M \times \mathbb{C} \arrow[r,"j'"]& T^*M(\log D) \\
\pi_2^{-1}(0) \arrow[r,hook,""] \arrow[u,hook',""]& \overline{\Lambda^\sharp} \arrow[u,hook', ""]&.
\end{tikzcd}
\end{equation}
Here  the map $j'$ sends
$(x, \xi, s)$ to $\left(x, j(\xi)-s \frac{df}{f}(x)\right)$. 

\begin{proposition}\label{analyticity:logimage}
$\overline{j(\Lambda)}$ is a complex analytic variety of dimension $n$.
\end{proposition}

\begin{proof}
Let $\dim Z=d$. The canonical map $j^\vee: TM(-\log D) \to TM$ is an isomorphism on $U$. We can define $\tilde Z$ the logarithmic Nash blowup of $Z$ as the closure of 
\[
\Big\{ (x,V)\in G_d(TM(-\log D)) \ | \ x\in Z_{reg}\cap U \text{ and } j^\vee(V)=T_xZ  \Big\}
\] 
in $G_d(TM(-\log D))$. We prove $\tilde{Z}$ is locally a monoidal transform, hence is complex analytic, and we follow the line of argument given in \cite[Theorem 1]{MR409462} where the ordinary Nash blowup is locally a monoidal transform is proved. Let $W\subset M$ be an open subset on which $\textup{Der}_M(-\log D)$ is trivialized by derivations $\delta_1,\ldots, \delta_n \in H^0(W,\textup{Der}_M(-\log D))$. Suppose $Z$ is defined by $r$ equations $g_1=\ldots = g_{r}=0$ in $W$ and $r\geq n-d$. Consider the logarithmic Jacobian matrix $J$ which is the $r\times n$ matrix $(\delta_j(g_i))$. If $Z\cap W$ is irreducible, then at least one $(n-d) \times (n-d)$ minor of $J$, say $\Delta$, is not identically zero on $Z\cap W$. We may assume that the rows of $\Delta$ are the first $n-d$ rows of $J$. Then $\tilde{Z}$ is locally the blowup of $Z\cap W$ along the ideal generated by all $(n-d)\times (n-d)$ minors of the first $n-d$ rows of $J$. If $Z\cap W$ has more than one irreducible compoments, we can use a partition of unity with respect to irreducible components of $Z\cap W$ to modify the argument in the irreducible case (cf. the proof of Theorem 1 in \cite{MR409462}). The detail is left to the reader. Let $p:\tilde{Z} \to Z$ be the natural projection, let $E$ be the tautological subbundle of $p^*TM(-\log D)$ and let $F$ be the quotient bundle (of rank $n-d$). The Cartesian diagram
\[
\begin{tikzcd}
F^\vee \arrow[d] \arrow[r] & \overline{j(\Lambda)} \arrow[d]\\
\tilde{Z} \arrow[r,"p"] & Z
\end{tikzcd}\
\]
shows that $\overline{j(\Lambda)}$ is the proper image of $F^\vee$. Therefore $\overline{j(\Lambda)}$ is analytic.
\end{proof}

\begin{definition}
\label{defn; logconormal}
We call the analytic variety $\overline{j(\Lambda)}$ the \textbf{logarithmic conormal space} of $Z$ along $D$. We denote it by $\Lambda^{log}_D$ or $\Lambda^{log}$ when there is no confusion.
\end{definition}

\begin{proposition}\label{sharpintolog}
$\overline{\Lambda^\sharp}$ is an irreducible component of $j'^{-1}\Lambda^{log}$ with multiplicity 1. Moreover, $\overline{\Lambda^\sharp}$ is the only irreducible component of $j'^{-1}\Lambda^{log}$ outside $D$.
\end{proposition}
\begin{proof}
By the definition of $j'$, one can see that $j'(\Lambda^\sharp) \subset j(\Lambda)$, hence $j'(\overline{\Lambda^\sharp})\subset \Lambda^{log}$. Over $U$, the map $j'$ makes $T^*M\times \mathbb{C}$ a rank 1 vector bundle over $T^*M(\log D)$ (the discussion below Proposition \ref{logsplitting} provides a coordinate free explanation for this fact). Therefore $\dim \overline{\Lambda^\sharp}=\dim j'^{-1}\Lambda^{log}=n+1$.
\end{proof}

Since there is no nonconstant global holomorphic function on a compact complex manifold, Ginzburg's sharp construction cannot be directly applied to many interesting cases (e.g. $M=\mathbb{P}^n$) and it causes some inconvenience in applications. The situation considered above is essentially a local one.

In the rest of this section we will generalise diagram \eqref{diagram; sharp} in the  global setup. We slightly modify earlier notations. From now on let $L$ be a holomorphic line bundle over $M$, let $f\in H^0(M,L)$ be a nontrivial global section and let $D=V(f)$ be the divisor cut by $f$. The sheaf of sections of $L$ is denoted by $\mathcal{L}\cong \mathscr{O}_M(D)$.

Following Atiyah \cite[\S 4]{MR86359} we recall the geometric definition of the  bundle of principal parts $P^1_ML$. Equivalently we describe its sheaf of sections $\mathcal{P}^1_M\mathcal{L}$.  
As a $\mathbb{C}_M$-module, $\mathcal{P}^1_M\mathcal{L}\cong \mathcal{L}\oplus (\mathcal{L}\otimes\Omega^1_M)$. 
The
$\mathscr{O}_M$-structure is given as follows. For $h\in \mathscr{O}_M(W), s\in \mathcal{L}(W), \alpha \in \mathcal{L}\otimes\Omega^1_M(W)$ where $W\subset M$ is any open subset,
\[
h\cdot (s,\alpha) := (hs, s\otimes dh + h\alpha).
\]
By definition $\mathcal{P}^1_M\mathcal{L}$ is locally free of rank $n+1$ and admits a short exact sequence
\begin{equation}
\label{diagram; P^1Lexact}
\begin{tikzcd}
0 \arrow{r} & \mathcal{L}\otimes\Omega^1_M  \arrow[r]  & \mathcal{P}^1_M\mathcal{L}  \arrow[r, ""]  & \mathcal{L} \arrow{r}  & 0 . 
\end{tikzcd}
\end{equation}   
One can see $P^1_M\mathcal{L} \cong \mathcal{L}\oplus (\mathcal{L}\otimes \Omega^1_M)$ as $\mathscr{O}_M$-modules if and only if there exists a connection $\nabla: \mathcal{L} \to \mathcal{L}\otimes \Omega^1_M$. When this happens, the isomorphism is explicitly given by $(s,\alpha) \mapsto (s, \alpha -\nabla s)$.

\begin{lemma}\label{lemma; logconnection} 
When $D$ is a free divisor, there exists a logarithmic connection $\nabla$ on $L$.
\end{lemma}
\begin{proof}
Let $\{U_\alpha\}$ be an open cover of $M$ such that $\mathcal{L}\vert_{U_\alpha} \cong \mathscr{O}_{U_\alpha}$ and let $f_\alpha \in \mathscr{O}_M(U_\alpha)$ be the image of $f\vert_{U_\alpha}$ under this isomorphism. The functions $\{f_\alpha\}$ allow us to embed $\mathcal{L}$ as a subsheaf of $\mathcal{M}_M$ the sheaf of meromorphic functions on $M$, i.e. $\mathcal{L}\vert_{U_\alpha}=\frac{1}{f_\alpha}\mathscr{O}_{U_\alpha}$ under this identification. The usual differentiation $d: \mathcal{M}_M \to \mathcal{M}_M \otimes \Omega^1_M$ restricted to $\mathcal{L}$ gives the desired logarithmic connection. Indeed, over $U_\alpha$ we have
\[
\nabla  \colon  \frac{g}{f_\alpha} \mapsto d(\frac{g}{f_\alpha} )= \frac{dg}{f_\alpha} - \frac{df_\alpha}{f^2_\alpha}g =\frac{1}{f_\alpha}\otimes (dg -g \frac{df_\alpha}{f_\alpha}) \/.
\]
for any $g\in \mathscr{O}_M(U_\alpha)$. Apparently $\nabla(\frac{g}{f_\alpha})$ is an element in $(\mathcal{L}\otimes\Omega_M^1(\log D))(U_\alpha)$. 
\end{proof}

\begin{proposition}\label{mapj'L}
When $D$ is a free divisor, the logarithmic connection $\nabla$ induces a $\mathscr{O}_M$-linear morphism $j'_L\colon \mathcal{P}^1_M\mathcal{L}\to \mathcal{L}\otimes\Omega^1_M(\log D)$.
\end{proposition}
\begin{proof}
Recall that $\mathcal{P}^1_M\mathcal{L}\cong \mathcal{L}\oplus (\mathcal{L}\otimes\Omega^1_M)$ as $\mathbb{C}_M$-modules.
As  in the proof of Lemma~\ref{lemma; logconnection}, for any $g\in \mathcal{L}(U_\alpha)$ and any $h\in \mathscr{O}_M(U_\alpha)$,  we define $j'_L$  by 
\[ 
j'_L(U_\alpha)\colon (\frac{g}{f_\alpha},  \frac{1}{f_\alpha}\otimes dh ) \mapsto  -d(\frac{g}{f_\alpha})+ j_\alpha (dh)\otimes \frac{1}{f_\alpha}  \/,
\]
where $j_\alpha$ is the inclusion $\Omega^1_{U_\alpha} \to  \Omega^1_{U_\alpha}(\log D\cap U_\alpha)$.  Using the $\mathscr{O}_M$-structure of $\mathcal{P}^1_M\mathcal{L}$ one can  check that $j'_L$ is $\mathscr{O}_M$-linear.
\end{proof}

We claim that $j'_L$ globalises $j'$ in the local setup. Indeed, since $\mathcal{L}\vert_{U_\alpha} \cong \mathscr{O}_{U_\alpha}$, the differentiation $d: \mathscr{O}_{U_\alpha} \to \Omega^1_{U_\alpha}$ induces a cannonical local splitting $\mathcal{P}^1_{U_\alpha}\mathscr{O}_{U_\alpha} \cong \mathscr{O}_{U_\alpha}\oplus\Omega^1_{U_\alpha}$ and henthforth  $\mathcal{P}^1_{U_\alpha}\mathcal{L}\cong \mathcal{L}\vert_{U_\alpha} \oplus (\mathcal{L}\vert_{U_\alpha}\otimes\Omega^1_{U_\alpha})$. The composition 
\begin{equation*}
\mathscr{O}_{U_\alpha}\oplus \Omega^1_{U_\alpha} \cong \mathcal{P}^1_{U_\alpha}\mathscr{O}_{U_\alpha} \cong \mathcal{P}^1_{U_\alpha}\mathcal{L} \to \mathcal{P}^1_{U_\alpha}\mathcal{L}(\log D) \to \mathcal{L}\vert_{U_\alpha}\otimes\Omega^1_{U_\alpha}(\log D) \cong \Omega^1_{U_\alpha}(\log D) \\
\end{equation*}
is given by
\begin{align*}
(g,\omega) &\mapsto (g,\omega+dg) \in  \mathcal{P}^1_{U_\alpha}\mathscr{O}_{U_\alpha} \\
&\mapsto \frac{1}{f_\alpha}\otimes (g,\omega+dg) \in \mathcal{P}^1_{U_\alpha}\mathcal{L} \\
&\mapsto (\frac{g}{f_\alpha}, \frac{\omega+dg}{f_\alpha}) \in \mathcal{P}^1_{U_\alpha}\mathcal{L}(\log D) \\
&\mapsto \frac{\omega+dg}{f_\alpha}-d(\frac{g}{f_\alpha}) = \frac{\omega}{f_\alpha}-\frac{gdf_\alpha}{f^2_\alpha} \in  \mathcal{L}\vert_{U_\alpha}\otimes\Omega^1_{U_\alpha}(\log D) \\
&\mapsto \omega-g\frac{df_\alpha}{f_\alpha} \in \Omega^1_{U_\alpha}(\log D),
\end{align*}
agreeing with the earlier definition of $j'$ in \eqref{diagram; sharp}.

In the global setting, it is sometimes convenient to consider a variant of the conormal space which we call jet conormal space first introduced in \cite{MR4749159}. For any local section $g$ of $L$, denote by $j^1(g)$ the local section of $P_M^1L$ induced by the $\mathbb{C}_M$-module inclusion $\mathcal{L} \to \mathcal{P}^1_M\mathcal{L}$. Let $Z$ be any irreducible subvariety of $M$ whose conormal space is $\Lambda=T^*_ZM$.

\begin{definition}\label{defn;jetconormal}
The \textbf{jet conormal space of $Z$} denoted by $J\Lambda$, is the closure of   
\[
\Big\{ (x,j^1(g)(x)) \ \vert \ x\in Z_{reg}, \ g(x)=0 \text{ and }   dg(x) \textup{ annihilate } T_xZ \Big\}
\]
in $T^*M\otimes L$.  
The fibre of $J\Lambda$ over $x\in Z$, which is the \textbf{space of jets conormal to $Z$ at $x$}, is denoted by $J\Lambda_x$. Clearly $J\Lambda_x=T^*_{Z,x}M\otimes L_x$.

Let $\gamma=\sum_\alpha m_\alpha \textup{Eu}^\vee_{Z_{\alpha}}$ and let $\Lambda_\alpha=T^*_{Z_\alpha}M$. The \textbf{jet characteristic cycle} of $\gamma$ is 
\[
\textup{JCC}(\gamma):= \sum_{\alpha}m_\alpha \cdot [J\Lambda_\alpha].
\]
\end{definition}

\begin{definition}\label{defn; logjetconormal}
Let $D$ be a free divisor. The \textbf{logarithmic jet conormal space to $Z$ along the divisor $D$}, denoted by $J\Lambda^{log}_D$, is the image closure of $J\Lambda$ under the $L$-twisted logarithmic map $j_L\colon T^*M\otimes L \to T^*M(\log D)\otimes L$. When $D$ is understood from the context we write it as $J\Lambda^{log}$. By Proposition \ref{analyticity:logimage}, $J\Lambda^{log}$ is analytic.
\end{definition}

We are ready to describe the global sharp operation. Suppose $Z\not\subset D$.
\begin{definition}
\label{defn; globaljetsharp}
Given a nontrivial $f\in H^0(M,L)$, the \textbf{global jet sharp construction} $\overline{J\Lambda^{\sharp}_f}$ is the closure of
\[
\Big\{ \Big(x,\xi+ t\cdot j^1(f)(x)\Big) \ \vert  \ x\in Z, \ \xi \in J\Lambda_x, \ t \in \mathbb{C} \ \Big\}
\]
in $P^1_ML$. When  $f$ is clear from the context, we write it as $\overline{J\Lambda^{\sharp}}$.
\end{definition}

The following results will be frequently used in local computations.

\begin{lemma}
$\mathcal{P}^1_M\mathscr{O}_M\cong \mathscr{O}_M\oplus \Omega^1_M$.
\end{lemma}

\begin{proof}
The splitting connection $\mathscr{O}_M\to \Omega^1_M$ is given by the usual derivation.
\end{proof}

Take any trivialization $\{(U_\alpha, f_\alpha)\}$  for $\mathcal{L}\cong \mathscr{O}_M(D)$ as in the proof of Proposition \ref{lemma; logconnection}.

\begin{proposition}\label{prop:trivialisation of JLambda}
On $U_\alpha$, $\overline{J\Lambda^\sharp}|_{U_\alpha}$ is isomorphic to $\overline{(T^*_{Z\cap U_\alpha}U_\alpha)^\sharp_{f_\alpha}}$.
\end{proposition}

\begin{proof}
Using the trivialization of $L$ on $U_\alpha$ we may consider the closure of 
\[
\Big\{ \Big(x,\xi+t\cdot j^1(f_\alpha)(x)\Big) \ \vert \ x\in Z\cap U_\alpha, \ \xi \in T^*_{Z,x}M, \ t \in \mathbb{C} \ \Big\}
\]
in $P^1_{U_\alpha}L\cong \mathbb{C}\times T^*U_\alpha$. Since $j^1(f_\alpha)=f_\alpha+df_\alpha$ in this decomposition, for any $x\in U\cap U_\alpha$
\begin{align*}
\xi +t\cdot j^1(f_\alpha)(x)&=\xi +t(f_\alpha(x)+df_\alpha(x)) \\
&=\xi + tf_\alpha(x)\cdot \frac{df_\alpha}{f_\alpha}(x) +tf_\alpha(x).
\end{align*}
So letting $s=tf_\alpha$ we return to the expression \eqref{equation:localginz} defining the local sharp operation.
\end{proof}

To summarise, we have the following diagram in the global setting. Proposition \ref{sharpintolog} implies that $j'_L$ maps $\overline{J\Lambda^\sharp}$ dominantly into $J\Lambda^{log}$, and $\overline{J\Lambda^\sharp}$ is one irreducible component of $j'^{-1}_L(J\Lambda^{log})$. The square on the right is only defined when $D$ is a free divisor.

\begin{equation}
\label{diagram; globaljetsharp}
\begin{tikzcd}
T^*M\otimes L \arrow[r, hook,"i_L"] \arrow[rr, bend left, "j_L"]& P^1_ML \arrow[r,"j'_L"]& T^*M(\log D)\otimes L \\
i^{-1}_L(\overline{J\Lambda^\sharp}) \arrow[r,hook,""] \arrow[u,hook',""]& \overline{J\Lambda^\sharp} \arrow[u,hook', ""] \arrow[r]& J\Lambda^{log}\arrow[u,hook',""].
\end{tikzcd}
\end{equation}

There is a more direct way to globalize Ginzburg's construction. The global section $f\in H^0(M,L)$ determines a meromorphic section $\frac{1}{f}$ of the line bundle $L^\vee$ whose polar locus is $D$, so that the expression $\frac{j^1(f)}{f}$ makes sense as a meromorphic section of $P^1_ML\otimes L^\vee$. 

\begin{definition}
\label{defn; globalsharp}
We define $\overline{\Lambda^\sharp_f}$, the \textbf{global sharp construction} for $\Lambda=T^*_ZM$ as the closure of 
\[
\Big\{ (x,\xi + s\frac{j^1(f)}{f}) \ \vert \ x\in Z\cap U, \ \xi\in T^*_{Z,x}M, \ s\in\mathbb{C} \Big\}
\]
in $P^1_ML\otimes L^\vee$.   When the section $f$ is clear from the context we write it as $\overline{\Lambda^\sharp}$. 
Clearly $\overline{\Lambda^\sharp}|_x=\overline{J\Lambda^\sharp}|_x\otimes L^\vee_x$ for any $x\in M$. 
\end{definition}

\begin{proposition}\label{global-Ginzburg}
Ginzburg's theorem becomes 
\[
\textup{CC}(\textup{Eu}_Z^\vee\cdot \ind_U)=i^*[\overline{\Lambda^{\sharp}}]=[i^{-1}(\overline{\Lambda^{\sharp}})]
\]
where $i$ is the natural inclusion $T^*M \to P^1_ML\otimes L^\vee$. Equivalently
\[
\textup{JCC}(\textup{Eu}_Z^\vee\cdot \ind_U)=i_L^*[\overline{J\Lambda^{\sharp}}]=[i_L^{-1}(\overline{J\Lambda^{\sharp}})].
\]
\end{proposition}

Finally we provide an alternative description of $\overline{\Lambda^\sharp_f}$ based on  graph embeddings which  will be used in the followup paper \cite{LZ24-3}.
Let $\iota_f, \iota_0\colon M\to L$  be the closed embeddings induced by $f$ and  of the zero section.
We denote their images  by $M_f=\iota_f(M)$ and $M_0=\iota_0(M)$ and let $\pi: L \to M$ be the natural projection.    

\begin{lemma}
$\mathscr{O}_L(M_0) \cong \pi^*\mathcal{L}$.
\end{lemma}
\begin{proof}
The first projection $p_1: L\times_M L \to L$ makes $L\times_M L$ a line bundle over $L$. The diagonal map given by $v \to (v, v)$ is a global section of $p_1$ whose set of zeros is $M$.
\end{proof}

By the proof of Lemma \ref{lemma; logconnection}, there exists a logarithmic connection $\pi^*\mathcal{L} \to \Omega^1_L(\log M_0)\otimes \pi^*\mathcal{L}$. Restricting the connection to $M_0$, we obtain a connection
\[
\nabla: \mathcal{L} \to \iota^*_0\Omega^1_L(\log M_0)\otimes\mathcal{L}. 
\] 
Restricting the residue map $\rho\colon \Omega^1_L(\log M_0)\to \mathscr{O}_{M_0}$ to $M$ by $\iota_0$ and tensor by $\mathcal{L}$ induces a short exact sequence on $M$:
\[
\begin{tikzcd}
0  \arrow{r} & \Omega^1_{M}\otimes\mathcal{L} \arrow{r} & \iota^*_0\Omega^1_L(\log M_0)\otimes \mathcal{L} \arrow[r,"\rho"] & \mathcal{L} \arrow{r} & 0.
\end{tikzcd}
\]

\begin{lemma}
$\rho\circ\nabla=-1$.
\end{lemma}
\begin{proof}
Let $t$ be a local equation of $M$ in $L$ and let $g$ be any locally defined holomorphic function on $M$, the computation 
\[
\rho\circ\nabla(\frac{g}{t}) = \rho(\frac{-gdt}{t^2}+\frac{dg}{t})=\frac{1}{t}\textup{res}_t(dg-\frac{gdt}{t})=-\frac{g}{t}
\]
implies the asserted equality.
\end{proof}

\begin{proposition}
$\iota^*_0\Omega^1_L(\log M)\otimes \mathcal{L} \cong \mathcal{P}^1_{M}\mathcal{L}$.
\end{proposition}
\begin{proof}
Define a morphism $\iota^*_0\Omega^1_L(\log M)\otimes \mathcal{L} \to \mathcal{P}^1_{M}\mathcal{L}$ by 
\[
\omega \mapsto \Big(\rho(\omega),\omega+\nabla\circ\rho(\omega)\Big)
\]
where $\omega$ is any local section of $\iota^*_0\Omega^1_L(\log M)\otimes \mathcal{L}$. The lemma above implies that $\omega+\nabla\circ\rho(\omega) \in \ker(\rho)=\Omega^1_{M}\otimes \mathcal{L}$ so the morphism is well-defined and $\mathbb{C}_{M}$-linear. It is straight forward to check it is also $\mathscr{O}_{M}$-linear using the definition of the $\mathscr{O}_{M}$-structure of $\mathcal{P}^1_{M}\mathcal{L}$. We leave it to the reader to check the morphism is an isomorphism.
\end{proof}

\begin{proposition}
$\Omega^1_L(\log M_0)\otimes \pi^*\mathcal{L} \cong \pi^*(\mathcal{P}^1_M\mathcal{L})$.
\end{proposition}

\begin{proof}
The $\mathbb{C}^*$-action along the direction of $L$ induces a $\mathbb{C}^*$-action on $\Omega^1_L(\log M_0)\otimes \pi^*\mathcal{L}$. Denote by $\mathscr{F}$ the invariant subsheaf of $\Omega^1_L(\log M_0)\otimes \pi^*\mathcal{L}$. Clearly $\mathscr{F}$ is completely determined by the restriction of $\Omega^1_L(\log M_0)\otimes \pi^*\mathcal{L}$ to $M$, i.e.
\[
\mathscr{F} \cong \pi^{-1}\Big(\iota_0^*\Omega^1_L(\log M_0)\otimes \mathcal{L}\Big)\cong  \pi^{-1}\mathcal{P}^1_{M}\mathcal{L}.
\]
On the other hand, one may easily check by a local computation that $\Omega^1_L(\log M_0)\otimes \pi^*\mathcal{L} \cong \mathscr{F} \otimes_{\pi^{-1}\mathscr{O}_M} \mathscr{O}_L  \cong \pi^*(\mathcal{P}^1_M\mathcal{L})$.
\end{proof}

\begin{corollary}
\label{cor; isoofbundles}
$\mathcal{P}^1_{M}\mathcal{L}\otimes \mathcal{L}^\vee \cong \iota^*_f\Omega^1_L(\log M_0)\cong \iota^*_0\Omega^1_L(\log M_f) $.
\end{corollary}

\begin{proof}
The first isomorphism follows directly from the proposition above. Since subtracting $f$ gives an affine isomorphism $(L,M_f,M_0) \to (L,M_0,M_{-f})$, and multiplication by $-1$ gives an isomorphism $(L,M_0,M_{-f}) \to (L,M_0,M_f)$, their combination is an isomorphism of triples $(L,M_f,M_0) \to (L,M_0,M_f)$ switching $M_0$ and $M_f$.
\end{proof}
 
Local computation also shows that
\begin{proposition}
The  differentials of the morphisms $\iota_f$ and $\iota_0$ induce morphisms  
$$d\iota_f \colon \iota_f^*T^*L(\log M_0) \to T^*M(\log D)\/;
\ 
 d\iota_0  \colon \iota_0^*T^*L(\log M_f)  \to T^*M(\log D) \/.
$$ 
Then under  the isomorphisms in Corollary~\ref{cor; isoofbundles} both $d\iota_f$  and  $d\iota_0$  are identified to the $L^\vee$-twist of the morphism $j'_L$  constructed in Proposition~\ref{mapj'L}.
\end{proposition}

Converting back to our earlier setting in Definition \ref{defn; globalsharp}. Let $\Lambda_{\iota_f}\subset T^*L$ (resp. $\Lambda_{\iota_0}\subset T^*L$) be the conormal space of $\iota_f(Z)$ (resp. $\iota_0(Z)$) in $L$. We denote by $j_0\colon T^*L\to T^*L(\log M_0)$ and $j_f\colon T^*L\to T^*L(\log M_f)$  the logarithmic maps.

\begin{proposition}
\label{prop; sharp=logcompletion}
$\overline{\Lambda^\sharp_f}\cong \overline{j_0(\Lambda_{\iota_f})}\cong \overline{j_f(\Lambda_{\iota_0})}$. 
\end{proposition}

\begin{proof}
Let $U_\alpha$    and $f_\alpha$ be as in the proof of Proposition \ref{lemma; logconnection}. We make the identifications $L|_{U_\alpha}\cong  U_\alpha\times \mathbb{C}_s$ and $\Lambda|_{U_\alpha}= T_{Z\cap U_\alpha}^*U_\alpha$.    
The graph embedding $\iota_f$ is given by
\[ 
U_\alpha \longrightarrow  U_\alpha \times \mathbb{C}_s \/; \quad 
x\mapsto (x, f_\alpha(x)). 
\]
 
Let $I$ be the  ideal of $Z\cap U_\alpha$ in $U_\alpha$, then the ideal of $Z\cap U_\alpha$ in $L|_{U_\alpha}$ is generated by $I$ and $s-f_\alpha$.  We have 
\[
{\Lambda_{\iota_f}}|_{U_\alpha}  =\overline{ \left\lbrace 
( x,f_\alpha(x), \xi-t\cdot df_\alpha(x) , tds) \ | \ 
(x,\xi)\in T_{Z\cap U_\alpha}^*U_{\alpha} \text{ and } t\in \mathbb{C}
\right\rbrace} \/.
\] 

Since $\Omega^1_L(\log M_0)(U_\alpha)$ is generated by $\Omega^1_{U_\alpha}$ and $\frac{ds}{s}$, the map ${j_0}|_{U_\alpha}$ is given by 
$$j_\alpha\colon T^*U_\alpha\times T^*\mathbb{C}_s \to T^*U_\alpha\times T^*\mathbb{C}_s(\log \{0\})\/; \quad (x,  \xi,l,  t ) \mapsto (x, \xi, l,  lt) \/.
$$ 
Thus we have:
\begin{align*}
\overline{j_0(\Lambda_{\iota_f})}|_{U_\alpha}&= 
\overline{j_\alpha({\Lambda_{\iota_f}}|_{U_\alpha})} \\
&=
\overline{
\left\lbrace
( x,\xi-tdf_\alpha(x)  , f_\alpha(x) t) \ | \ 
(x,\xi)\in T_{Z\cap U_\alpha}^*U_{\alpha} \text{ and } t\in \mathbb{C}
\right\rbrace} \\
&=
\overline{
\left\lbrace
( x,\xi-s\frac{df_\alpha(x)}{f_\alpha(x)}  , s) \ | \ 
(x,\xi)\in T_{Z\cap U_\alpha}^*U_{\alpha} \text{ and } s\in \mathbb{C}^*
\right\rbrace} \\
&=\overline{(\Lambda\vert_{U_\alpha})_{f_\alpha}^\sharp}  \/. 
\end{align*}
This gives the first isomorphism.
The isomorphism $(L,M_0,M_f) \to (L,M_f,M_0)$ in the proof of Corollary \ref{cor; isoofbundles} gives the second isomorphism.
\end{proof}

Let $k_{0}$ be the $\pi^*L$-twisted logarithmic map along $M_0$
\[
T^*L\otimes_L  \pi^*L\to 
T^*L(\log M_0)\otimes_L \pi^*L \cong 
\pi^*P_M^1L \/,
\]
and let $k_f$ be the $\pi^*L$-twisted logarithmic map along $M_f$
\[
T^*L\otimes_L  \pi^*L\to 
T^*L(\log M_f)\otimes_L \pi^*L \cong 
\pi^*P_M^1L \/.
\]
Let $J\Lambda_{\iota_f}\subset T^*L\otimes \pi^*L$ (resp. $J\Lambda_{\iota_0}\subset T^*L\otimes \pi^*L$) be the jet conormal space of $\iota_f(Z)$ (resp. $\iota_0(Z)$) in $L$. The jet parallel of Proposition \ref{prop; sharp=logcompletion} is the following.
\begin{proposition}
\label{prop; Jsharp=logcompletion}
$\overline{J\Lambda^\sharp_f}\cong \overline{k_0(J\Lambda_{\iota_f})}\cong \overline{k_f(J\Lambda_{\iota_0})}$. 
\end{proposition}

\section{Log Transversality and Chern Classes}\label{sec:formal consequence}
Let $D$ be a free divisor on $M$ defined by $f\in H^0(M,L)$. The maps obtained by tensoring   $L^{\vee}$ to $i_L,j_L,j'_L$ (see diagram \eqref{diagram; globaljetsharp}) will be denoted by $i,j,j'$ in the sequel, and they agree with the maps with the same names in diagram \eqref{diagram; sharp} on any chart trivialising $L$. We will use the notation $\Lambda_\alpha$ to indicate $T^*_{Z_\alpha}M$ when $Z_\alpha$ is an irreducible subvariety of $M$. Let $\gamma$ be a constructible function on $M$ with characteristic cycle 
\begin{equation*}\label{decom2}
\textup{CC}(\gamma)=\sum_\alpha n_\alpha[T^*_{Z_\alpha}M] + \sum_\beta m_\beta[T^*_{Z_\beta}M]=\sum_\alpha n_\alpha[\Lambda_\alpha] + \sum_\beta m_\beta[\Lambda_\beta]
\end{equation*}
where the index is arranged such that each $Z_\alpha \not\subset D$ and each $Z_\beta \subset D$.

\begin{definition}
\label{defn; logtransverse}
We say $\gamma$ is \textbf{log transverse to $D$} if $\overline{\Lambda^\sharp_\alpha}$ is the only irreducible component of $j'^{-1}(\Lambda^{log}_\alpha)$ for every $\alpha$. 
We say $\gamma$ is \textbf{weakly log transverse to $D$} if $\overline{\Lambda^\sharp_\alpha}$ is the only irreducible component of $j'^{-1}(\Lambda^{log}_\alpha)$ outside $T^*M$ for every $\alpha$.  When $\gamma=\textup{Eu}^\vee_Z$ with $Z\not\subset D$, we simply say $Z$ is (weakly) log transverse to $D$.
\end{definition}

\begin{remark}
The map $j'$ and $j$ are local complete intersection morphisms in the sense of \cite{MR1644323} \S 6.6. The dimension of every irreducible component of $j'^{-1}(\Lambda^{log})$ is no less than $n+1$, and the dimension of every irreducible component of $j^{-1}(\Lambda^{log})$ is no less than $n$. 
\end{remark}

\begin{remark}\label{remark:logtrans}
The definition easily implies that $Z$ is log transverse to $D$ if and only if $Z$ is weakly log transverse to $D$ and $\dim j^{-1}(\Lambda^{log})\leq n$. 
\end{remark}

The definition of $j'_L$ (hence $j'$) clearly depends on the choice of $f\in H^0(M,L)$ (note that the splitting of Proposition \ref{logsplitting} depends on the embedding $\mathcal{L} \to \mathcal{M}_M$). However, as the next proposition will show, the definition of (weakly) log transversality is independent of the choice of $f$. This will be particularly important when we want to utilise properties of $D$ which relies on the choice of local equations, e.g. strongly Euler homogeneity.

\begin{proposition}
The definition of (weakly) log trasnsversality depends only on $D$ and $\gamma$, but not on the equation of $D$.
\end{proposition}

\begin{proof}
Suppose $f_1, f_2\in H^0(M,L)$ both define the divisor $D$, i.e. $f_1=uf_2$ for some $u\in H^0(M,\mathscr{O}_M^*)$, and the maps $j'$ constructed out of $f_1,f_2$ are denoted by $j'_{f_1},j'_{f_2}$ respectively. Let $\pi_1,\pi_2$ be the compositions of $P^1_ML\otimes L^\vee \to M\times \mathbb{C}$ and the two projections respectively. Define an affine automorphism $g$ of $P^1_ML\otimes L^\vee$ by 
\[
g(v)=v-\pi_2(v)\cdot i\Big(\frac{du}{u}(\pi_1(v))\Big).
\]
Examining diagram \eqref{diagram; sharp} we obtain $j'_{f_1}=j'_{f_2}\circ g$. Since $g$ sends the subset $\{\pi_2\neq 0\}$ isomorphically to itself, $\overline{\Lambda^\sharp_{f_1}}$ is the only irreducible component of $j'^{-1}_{f_1}(\Lambda^{log})$ 
(in $\{\pi_2\neq 0\}$) if and only if $\overline{\Lambda^\sharp_{f_2}}$ is the only irreducible component of $j'^{-1}_{f_2}(\Lambda^{log})$ (in $\{\pi_2\neq 0\}$).
\end{proof}

\begin{remark}\label{remark:choice of function}
Let $\{U_\alpha\}$ be an open cover of $M$  such that $\mathcal{L}\vert_{U_\alpha}\cong \mathscr{O}_{U_\alpha}$. It is clear that $Z$ is (weakly) log transverse to $D$ if and only if $Z\cap U_\alpha$ is (weakly) log transverse to $D\cap U_\alpha$ for every $\alpha$. The discussion above implies that we may use any defining equation of $D\cap U_\alpha$ in $U_\alpha$ for the purpose of checking (weakly) log transversality.
\end{remark}

Suppose $\gamma$ is weakly log transverse to $D$. Proposition \ref{sharpintolog} implies 
\[
j'^*\sum_\alpha n_\alpha[\Lambda^{log}_\alpha]=\sum_\alpha n_\alpha[\overline{\Lambda^{\sharp}_\alpha}]+i_*\Theta
\] 
where $\Theta\in H_{2n+2}(T^*M)$. Since $T^*M$ has trivial normal bundle in $P^1_ML\otimes L^\vee$, combining Proposition \ref{global-Ginzburg} we have
\begin{equation}
\label{relation; CCnonequivariant}
[\textup{CC}(\gamma \cdot \ind_U)] =i^*\sum_\alpha n_\alpha[\overline{\Lambda^{\sharp}_\alpha}]=\sum_\alpha n_\alpha j^*[\Lambda^{log}_\alpha] \ \in \ H_{2n}\Big(\bigcup_\alpha j^{-1}\Lambda^{log}_\alpha\Big) \/.
\end{equation}

With further assumption that $\gamma$ is log transverse to $D$, we may consider characteristic cycles in equivariant homology groups. Recall that the complex torus $T=\mathbb{C}^*$  acts trivially on $M$ and acts by scalar multiplication on vector bundles on $M$. Proposition \ref{sharpintolog} implies
\[
j'^*_T\sum_\alpha n_a[\Lambda^{log}_\alpha]_T=\sum_\alpha [\overline{\Lambda^\sharp_\alpha}]_T.
\]
Applying the equivariant version of Proposition \ref{global-Ginzburg} we obtain
\begin{equation}
\label{relation; CCequivariant}
[\textup{CC}(\gamma \cdot \ind_U)]_T =\sum_\alpha n_\alpha j_T^*[\Lambda^{log}_\alpha]_T \ \in \ H^T_{2n}\Big(\bigcup_\alpha j^{-1}\Lambda^{log}_\alpha\Big)  \/.
\end{equation}

Our expectation to compute $\textup{CC}(\gamma\cdot\ind_U)$ using logarithmic conormal spaces is reflected in the following   formula. 
\begin{theorem}\label{theo;logCCtochern}
Let $p: \mathbf{P}(T^*M(\log D)\oplus 1)\to M$ be the natural projection and  $\xi$ be the tautological line subbundle of $\mathbf{P}(T^*M(\log D)\oplus 1)$. 
\begin{enumerate}[label=(\roman*)]
\item If $\gamma$ is weakly log transverse to $D$, then for the dimension $0$ piece we have
\begin{equation*}
\{c_*(\gamma\cdot \ind_U)\}_0 = (-1)^{n} \sum_\alpha n_\alpha\cdot  \left( \sum_{k=0}^nc_{n-k}\Big(TM(-\log D) \Big)\cap     p_*\Big(c_1(\xi)^{k} \cap [\mathbf{P}(\Lambda^{log}_\alpha \oplus 1)]\Big)\right) \/.
\end{equation*}  
\item If $\gamma$ is log transverse to $D$, then
\begin{equation*}
c_*(\gamma\cdot \ind_U) = (-1)^{n}c\Big(TM(-\log D) \Big)\cap \sum_\alpha n_\alpha p_*\Big(c(\xi^{\vee})^{-1} \cap [\mathbf{P}(\Lambda^{log}_\alpha \oplus 1)]\Big) \/.
\end{equation*}
\end{enumerate}
\end{theorem}
 
\begin{proof}
Let $s: M \to T^*M$ be the zero section embedding. Applying $s^*$ to equality \eqref{relation; CCnonequivariant} we obtain (i).  According to Proposition 3.3 and 2.7 in \cite{AMSS23}, applying $s^*_T$ to equality \eqref{relation; CCequivariant}, dehomogenising, and dualising the homology class we obtain (ii). 
\end{proof}

Formula (ii) is the same as the formula of Corollary \ref{c1}. The point is that we do not ask $D$ to be simple normal crossing, but only impose the log transversality condition on $\gamma$.

\begin{corollary}\label{coro; chernclass}
If $M$ is compact and weakly log transverse to $D$, then
\[
\chi(U) = \int_M c_n(TM(-\log D))\cap [M]\/.
\]
If $M$ (not necessarily compact) is log transverse to $D$, then
\[
c_*(\ind_U) = c(TM(-\log D))\cap [M]\/.
\]
\end{corollary}

\section{Geometric Tests for log Transversality} 
\label{geo-test}

Testing log transversality is in general difficult. Before giving any general criterion, we start from the following result to whet the reader's appetite.

\begin{proposition} 
\label{prop; lineartype}
$M$ is log transverse to $D$, if $D$ is a free divisor with  Jacobian ideal of linear type. 
\end{proposition}

Proposition \ref{prop; lineartype} combined with Corollary \ref{coro; chernclass} gives a short and transparent proof of the main theorem of \cite{MR3830794}. 
 
\begin{proof}
Suppose the divisor $D$ is defined by $f=0$ where $f\in H^0(M,L)$ and $\mathcal{L}\cong \mathscr{O}(D)$. The global section $j^1(f) \in H^0(M,P^1_ML)$ corresponds to a morphism of sheaves
\[
(\mathcal{P}^1_M\mathcal{L})^\vee \to \mathscr{O}_M.
\]
A local computation shows that the image of the morphism is $\mathcal{J}_D$ (see Definition \ref{definition:Jacobianlineartype}) and the kernel of the morphism is $j'^\vee:\textup{Der}_M(-\log D)\otimes \mathcal{L}^\vee \to (\mathcal{P}^1_M\mathcal{L})^\vee$ (see diagram \eqref{diagram; globaljetsharp}). Therefore there is a short exact sequence
\[
\begin{tikzcd}
0  \arrow{r} & \textup{Der}_M(-\log D)\otimes \mathcal{L}^\vee \arrow{r} & (\mathcal{P}^1_M\mathcal{L})^\vee \arrow{r} & \mathcal{J}_D \arrow{r} & 0.
\end{tikzcd}
\]
This sequence induces a surjective morphism
\[
\textup{Sym}^\bullet_{\mathscr{O}_M}(\mathcal{P}^1_M\mathcal{L})^\vee \to \textup{Sym}^\bullet_{\mathscr{O}_M}(\mathcal{J}_D)
\]
whose kernel is generated by the image of $\textup{Der}_M(-\log D)\otimes \mathcal{L}^\vee$ in $\textup{Sym}^\bullet_{\mathscr{O}_M}(\mathcal{P}^1_M\mathcal{L})^\vee$. Since $\textup{Der}_M(-\log D)\otimes \mathcal{L}^\vee$ defines the zero section of $T^*M\otimes L$, we have 
\[
j'^{-1}_L(M)=j'^{-1}_L(J\Lambda^{log})=\textup{Spec}\Big(\textup{Sym}^\bullet_{\mathscr{O}_M}(\mathcal{J}_D)\Big)=\textup{Spec}\Big(\textup{Rees}^\bullet_{\mathscr{O}_M}(\mathcal{J}_D)\Big)
\]
where $M=JT^*_MM=J\Lambda=J\Lambda^{log}$ is the zero section of both $T^*M\otimes L$ and $T^*M(\log D)\otimes L$, and the last equation is the only place in the proof where the linear type Jacobian assumption is used. Finally we notice that
\[
\textup{Spec}\Big(\textup{Rees}^\bullet_{\mathscr{O}_M}(\mathcal{J}_D)\Big)=\overline{J\Lambda^\sharp}
\]
by Definition \ref{defn; globaljetsharp} since the closure of 
\[
\Big\{ \Big(x, t\cdot j^1(f)(x)\Big) \ \vert  \ x\in M,  \ t \in \mathbb{C} \ \Big\}
\]
in $P^1_ML$ is exactly the total space of the deformation to the normal cone associated to the global section $j^1(f)$ (see \cite{MR1644323} Remark 5.1.1). To summarise, we have proved that $j'^{-1}_L(J\Lambda^{log})=\overline{J\Lambda^\sharp}$ as subschemes of $P^1_ML$ which is equivalent to $j'^{-1}(\Lambda^{log})=\overline{\Lambda^\sharp}$.

\end{proof}

Given $Z\not\subset D$ and $\Lambda=T^*_ZM$, the agreement of $j'^{-1}(\Lambda^{log})$ and $\overline{\Lambda^\sharp}$ as subschemes of $P^1_ML\otimes L^\vee$ is the best case but this may rarely happen. However, the (weak) log transversality condition only concerns the support $|j'^{-1}(\Lambda^{log})|$. Recall that $\pi_1,\pi_2$ are the compositions of $P^1_ML\otimes L^\vee \to M\times \mathbb{C}$ and the two projections respectively.

\begin{definition}\label{definition:weaklylog at p}
We say that $Z$ is \textbf{weakly log transverse to $D$ at  $p\in Z\cap D$} if
\[
|j'^{-1}(\Lambda^{log})\cap \pi_1^{-1}(p)| \subset  \pi^{-1}_2(0)=T^*M  \/.
\]
\end{definition}

By Proposition \ref{sharpintolog}, $Z$ is weakly log transverse to $D$ if and only if $Z$ is weakly log transverse to $D$ at every point $p\in Z\cap D$. For $p\in Z\cap D$, choose an analytic neighborhood $\mathcal{U}_p$ of $p$ such that $L\vert_{\mathcal{U}_p}$ is trivial. In such situation $j'$ is given by the formula in diagram \eqref{diagram; sharp}. Note that by Remark \ref{remark:choice of function} we may choose any convenient local equation of $D$ at $p$. 

Following \cite{MR586450} it is usually convenient to choose an expression of the map $j'$ using the logarithmic stratification. Let $S$ be the logarithmic stratum of $D$ containing $p$ and $\dim S=n-k$. By the triviality lemma (\cite{MR586450} (3.6)), the neighborhood $\mathcal{U}_p$ can be chosen together with coordinate functions $x_1,\ldots,x_n$ with properties
\begin{enumerate}[label=(\roman*)]
\item the equation $h(x)$ of $D\cap \mathcal{U}_p$ in $\mathcal{U}_p$ involves only $x_1,\ldots,x_{k}$;
\item $D_\alpha \cap \mathcal{U}_p=\{x_1=\ldots=x_{k}=0\}$;
\item $D$ is strongly Euler homogeneous at $p$ if and only if the function $h(x_1,\ldots,x_k)$ in (i) can be chosen to be strongly Euler homogeneous according to Lemma 3.2 in \cite{MR2231201}.
\end{enumerate} 

Geometrically this means that $(\mathcal{U}_p,D,p) \cong (\mathbb{C}^n,D'\times \mathbb{C}^{n-k},0)$ for some divisor germ $(D',0) \subset(\mathbb{C}^k,0)$. strongly Euler homogeneities of $(D,p)$ and $(D',0)$ are equivalent.

Therefore, we can assume the basis $(\delta_1,\ldots,\delta_n)$ for $\textup{Der}_{\mathcal{U}_x}(-\log D\cap \mathcal{U}_x)$ takes the form
\[
\left\lbrace \delta_1=\sum_{j=1}^k \delta_{1,j}\partial_{x_j},\ \ldots, \ \delta_k=\sum_{j=1}^k \delta_{k,j}\partial_{x_j}, \  \delta_{k+1}=\partial_{x_{k+1}},\  \ldots , \  \delta_n=\partial_{x_{n}} \right\rbrace \/
\]
where $\delta_{i,j}\in \mathcal{O}_{M,p}$ are functions in $x_1, \cdots ,x_k$ only. Let $\omega_1,\ldots,\omega_n\in \Omega^1_{M,p}(\log D)$ be the dual basis of $\delta_1,\ldots,\delta_n$. In  particular, $\omega_i=dx_i$ for $i=k+1,\ldots, n$. We identify $T^*_pS$ with the $\mathbb{C}$-vector space spanned by $\omega_{k+1},\ldots,\omega_n$ and $T^*_pM(\log D)$ with the $\mathbb{C}$-vector space spanned by $\omega_1,\ldots,\omega_n$. Using these bases, it is easy to see that the linear map $j_p: T^*_pM \to T^*_pM(\log D)$ induced by $j$ is
\[
j_p(\xi)=\sum_{i=k+1}^n \delta_{i}(\xi)\cdot\omega_{i}. 
\]
In coordinate, the map is given by $j_p(\xi_1,\ldots,\xi_n)=(0,\ldots,0,\xi_{k+1},\ldots,\xi_n)$. Equivalently, regarding $T^*S$ as a subspace of $T^*M(\log D)$, the map $j\vert_S$ can be identified with the natural restriction map $T^*M\vert_S \to T^*S$. Similarly, the map $j'_p: T^*_pM\times \mathbb{C} \to T^*_pM(\log D)$ induced by $j'$ (defined using $h$) is 
\begin{equation}\label{equation:j'}
j'_p(\xi_1,\ldots,\xi_n,s)=\sum_{i=k+1}^n \xi_i\omega_{i} -s\cdot\sum_{i=1}^k \delta_i(\frac{dh}{h})(p)\cdot\omega_i .
\end{equation}

The following proposition implies that it is necessary to impose strongly Euler homogeneity condition on $D$ when studying weakly log transversality.

\begin{proposition}
\label{prop; whyEH}
If the divisor $D$ is not strongly Euler homogeneous at $p$, then any subvariety $Z\not\subset D$ containing $p$ is not weakly log transverse to $D$ at $p$. 
\end{proposition}

\begin{proof}
Set things up as above and let $\Lambda=T^*_ZM$. By the assumption, $h$ is not an Euler homogeneous equation, therefore $\delta_i(h)\in \mathfrak{m}_ph$ for any $i=1, \cdots ,k$. Equivalently $\frac{\delta_i(h)}{h}(p)=0$ and $j_p'(\xi_1,\ldots,\xi_n,s)=(0,\ldots,0,\xi_{k+1},\ldots,\xi_n)$. Thus
\[
j_p'(0,\ldots,0,1)=(0,\ldots,0)\in \Lambda^{log}.
\]
Since $\pi_2(0,\ldots,0,1)=1$, $Z$ is not weakly log transverse to $D$ at $p$.
\end{proof}

Recently Rodr\'{i}guez established a criterion for strong Euler homogeneity of any divisor germ $(D,0)$. When $(D,0)$ is free, it is easy to see that his criterion \cite[Proposition 2.6]{Rodriguez25a} coincides with ours. Indeed, with our notation, his criterion for strong Euler homogeniety at $p\in D$ is $\textup{rank} j'_p= \textup{rank} j_p+1$, which is equivalent to $\ker j'_p=\ker j_p$. The last equality is exactly the definition of weak log transversality of $M$ to $D$.

In the rest of the section we will study (weak) log transversality in two diametrically opposed situations. In the first, the divisor $D$ is strongly Euler homogeneous which is logically the weakest assumption we must impose, but the variety $Z$ is in generic position. In the second, $D$ is normal crossing but the variety $Z\not\subset D$ is arbitrary.

\begin{lemma}\label{lemma:noncharacteristic}
All logarithmic strata of $D$ are non-charactersitic with respect to $\Lambda=T^*_ZM$ (see Definition \ref{defn; lognonchar}) if and only if $j^{-1}(M)\cap \Lambda \subset M=T^*_MM$.
\end{lemma}
\begin{proof}
Let $(p,\xi) \in j^{-1}(M)\cap \Lambda$ and let $S$ be the logarithmic stratum of $D$ containing $p$. Since $j\vert_S$ can be identified with $T^*M\vert_S \to T^*S$, $j_p(\xi)=0$ implies that $\xi\in T^*_SM$. If the inclusion $S\to M$ is non-characteristic with respect to $\Lambda$, then $\xi=0$. The conclusion of the lemma follows easily from this observation.
\end{proof}

\begin{remark}
The non-characteristic assumption is satisfied if a Whitney stratification $\{Z_i\}$ of $Z$ is transverse to all logarithmic strata of $D$, i.e. $T_pZ_i+T_pD_\alpha=T_pM$ for any $p\in Z_i\cap D_\alpha$. In particular if $Z\neq M$ then it should avoid the $0$-dimensional strata of $D$, which are the most singular parts of $D$. 
\end{remark} 

\begin{remark}
We note that our non-characteristic condition differs from the one discussed in \cite{MR3706222}.  The latter is formulated with respect to  a Whitney stratification of $D$, whereas ours concerns the logarithmic  stratification of $D$. In general there is no direct comparison between Whitney stratifications  of $D$ and its logarithmic stratification.  
\end{remark}

\begin{proposition}\label{prop; genericZ}
Let $D$ be a strongly Euler homogeneous free divisor such that all logarithmic strata of $D$ are non-charactersitic with respect to $\Lambda=T^*_ZM$. Then $Z$ is weakly log transverse to $D$.
\end{proposition}

\begin{proof}
Since $\Lambda$ is conical, Lemma \ref{lemma:noncharacteristic} implies that $j$ induces a map $\mathbf{P}(\Lambda) \to \mathbf{P}(T^*M(\log D))$. Because the image of $\mathbf{P}(\Lambda)$ in $\mathbf{P}(T^*M(\log D))$ is closed, we conclude that $j(\Lambda)$ is closed in $T^*M(\log D)$, i.e. $\Lambda^{log} =j(\Lambda)$.

Suppose $(p,\xi,s) \in j'^{-1}j(\Lambda)$. Let $S$ be the logarithmic stratum of $D$ containing $p$ and $\dim S=n-k$. Thus there exists $(p,\eta)\in \Lambda$ such that $j'(p,\xi,s)=j(p,\eta)$. 
\begin{equation*}
-s\sum_{i=1}^k \delta_i(\frac{dh}{h})(p)\omega_i + \sum_{i=k+1}^n \xi_i\omega_i = \sum_{i=k+1}^{n}\eta_i\omega_i.
\end{equation*}

Since $\omega_1,\ldots,\omega_n$ are linearly independent, $s\delta_i(\frac{dh}{h})(p)=0$ for $i=1,\ldots k$. Since $h$ is strongly Euler homogeneous, the existence of $\chi\in \textup{Der}_{M,p}(-\log D)$ such that $\chi(h)=h$ implies that $\frac{\delta_(h)}{h}(p)$ cannot be zero for all $i$. Therefore $s=0$.
\end{proof}

\begin{lemma}\label{lemma:inverse image Lambda}
Let $S$ be a logarithmic stratum of any free divisor $D$, and let $Z\not\subset D$ be a subvariety whose conormal space is $\Lambda$. Then $\dim T^*M\vert_S \cap j^{-1}j(\Lambda) \leq n$. In particular, if $D$ is holonomic, $\dim j^{-1}j(\Lambda)\leq n$.
\end{lemma}

\begin{proof}
We explained that $j\vert_S$ can be identified with the natrual restriction $T^*M\vert_S \to T^*S$. Since the rank of $\ker(j\vert_S)$ is $n-\dim S$, it is enough to prove that $\dim T^*S \cap j(\Lambda\vert_S) \leq \dim S$. Take a stratification $\{S_i\}$ of $S\cap Z$ such that all pairs $((Z\setminus D)_{reg},S_i)$ satisfy the Whitney A condition. The Whiney condition guarantees that the image of $\Lambda\vert_{S_i}$ under the restriction map $T^*M\vert_{S_i} \to T^*S_i$ is contained in the zero section of $T^*S_i$. Since there is a factorisation
\[
T^*M\vert_{S_i} \to T^*S\vert_{S_i} \to T^*S_i,
\]
the image of $\Lambda\vert_{S_i}$ under the map $T^*M\vert_{S_i} \to T^*S\vert_{S_i}$ is contained in $T^*_{S_i}S$, so $\dim j(\Lambda\vert_{S_i}) \leq \dim T^*_{S_i}S=\dim S$.
\end{proof}

\begin{theorem}\label{theo; generic and holonomic}
If all logarithmic strata of $D$ are non-characteristic with respect to $\Lambda=T^*_ZM$ and $D$ is holonomic, then $Z$ is log transverse to $D$.
\end{theorem}
\begin{proof}
We explained in the proof of Proposition \ref{prop; genericZ} that the assumption implies that $j(\Lambda)=\Lambda^{log}$. The theorem follows from Proposition \ref{prop; genericZ} and Lemma \ref{lemma:inverse image Lambda}.
\end{proof}

\begin{corollary}
\label{coro; equiSEH}
$M$ is weakly log transverse to $D$  if and only if $D$ is strongly Euler homogeneous. $M$ is log transverse to $D$ if and only if $D$ is strongly Euler homogeneous and holonomic.
\end{corollary} 
\begin{proof}
Since all logarithmic strata of $D$ are non-characteristic with respect to $M=T^*_MM$, the first equivalence follows from Proposition \ref{prop; whyEH} and Proposition \ref{prop; genericZ}. If $M$ is log transverse to $D$, then by Remark \ref{remark:logtrans} $\dim j^{-1}(M)=n$. By \cite[(3.18) Proposition]{MR586450} this condition is equivalent to the holonomicity of $D$.
\end{proof}

\begin{corollary}\label{coro; SEHeulercharacteristic}
Let  $D$ be a  strongly Euler homogeneous free divisor in a complex compact manifold $M$ and $U=M\setminus D$ be its complement. Then 
\[
\chi(U)= \int_M c(\textup{Der}_M(-\log D))\cap [M] \/.
\]
\end{corollary}

\begin{proof}
Combine Corollary \ref{coro; chernclass} and Corollary \ref{coro; equiSEH}.
\end{proof}

When $M$ is a homogeneous space, for any fixed  logarithmic stratum $D_\alpha$ the transversality theorem of Kleiman \cite{Kleiman74} states that a general translate of $Z$ intersects $D_\alpha$ transversely. If $M$ is compact and $D$ is holonomic then $D$ only has finitely many logarithmic strata. 
\begin{corollary}
\label{coro; generaltranslate}
Let $M$ be a compact homogeneous spac, let $D$ be a holonomic and strongly Euler homogeneous free divisor and let $Z$ be a closed subvariety of $M$. Then a general translate of $Z$  is log transverse to $D$.
\end{corollary}

Next we consider the case where $D$ is a normal crossing divisor. Using a result proved in \cite{BMM00}, this case turns out to be quite simple.

\begin{theorem}
\label{theo; NCimpliesLT}
Let $D$ be a normal crossing divisor and let $Z$ be irreducible such that $Z\not\subset D$. Then $Z$ is log transverse to $D$.
\end{theorem}
\begin{proof} 
We prove that $Z$ is log transverse to $D$ at any $p\in Z\cap D$. Let $\mathcal{U}$ be a small neighbourhood of $p$ in which the divisor $D$ is cut by   $x_1 \cdots  x_k=0$. Set $\Lambda=T_{Z\cap \mathcal{U}}^*\mathcal{U}$, we can form the following multi-parameter sharp construction introduced in \cite{KK79}.
\[
 \Lambda_k^\sharp :=
\overline{
\left\lbrace
\left(p, \xi+s_1\frac{dt_1}{t_1}+\cdots +s_k\frac{dt_k}{t_k}, s_1, \cdots ,s_k\right)\Big\vert x\notin D \text{ and } (p, \xi)\in \Lambda
\right\rbrace}\subset T^*\mathcal{U}\times \mathbb{C}^k.
\]
Similar to the local map $j'$ defined in \eqref{diagram; sharp} we define $\rho\colon T^*\mathcal{U}\times \mathbb{C}^k \to  T^*\mathcal{U}(\log D)$ that sends $(x, \xi, s_1, \cdots ,s_k)$ to $(x, j(\xi)-\sum_{i=1}^r s_i\frac{dt_i}{t_i})$.  
By definition  $\Lambda^{\log}$ is the closure of $\rho(\Lambda_k^\sharp )$. 
The map $\rho$ is a surjective map of vector bundles, therefore $\dim \rho^{-1}(\Lambda^{\log})= n+k=\dim \Lambda_k^\sharp $. Since both spaces are irreducible, we have $\Lambda_k^\sharp=\rho^{-1}(\Lambda^{\log})\/.$  Consider the following diagram.

\[
\begin{tikzcd}
\overline{\Lambda^\sharp} \arrow{r}\arrow{d} & \Lambda_k^\sharp \arrow{r} \arrow{d}  & \Lambda^{\log} \arrow{d} \\
T^*\mathcal{U}\times \mathbb{C} \arrow[r, "i"] &  T^*\mathcal{U}\times \mathbb{C}^r \arrow[r, "\rho"]  & T^*\mathcal{U}(\log D)
\end{tikzcd}
\]
Here $i\colon (x, \xi, t)\mapsto (x, \xi, t, \cdots ,t)$ is the small diagonal embedding. By what we have shown above, the right square is Cartesian. The left square is also Cartesian by \cite[Corollary 1]{BMM00}, where an elegant conceptional proof involving only basic analytic geometry and the definition of Lagrangian variety is provided.
Since $j'=\rho\circ i$,  we have 
 ${j'}^{-1}(\Lambda^{\log})=\overline{\Lambda^\sharp}$.  
\end{proof}

Finally we give a general criterion which says that weak log transversality can be tested by analytic paths.

Let $D$ be a strongly Euler homogeneous free divisor and let $p\in D$ be any point. Let $D_\alpha$ be the logarithmic stratum containing $p$ and suppose $\dim D_\alpha=n-k$. Choose an analytic neighbourhood $\mathcal{U}_p$ of $p$ with coordinate functions $x_1,\ldots, x_n$ such that $D_\alpha \cap \mathcal{U}_p$ is given by $x_1=\ldots = x_k=0$ and the equation of $D\cap \mathcal{U}_p$ is given by a strongly Euler homogeneous equation $h$ involving only the variables $x_1,\ldots,x_k$. Let $\chi$ be a Euler vector field for $h$, i.e. $\chi(h)=h$. A basis $\{\delta_1,\ldots,\delta_n\}$ of $\textup{Der}_{M,p}(-\log D)$ and a dual basis $\{\omega_1,\ldots,\omega_n\}$ of $\Omega^1_{M,p}(-\log D)$ can be chosen such that
\begin{enumerate}[label=(\roman*)]
\item $\delta_1=\chi$ and $\omega_1=\frac{dh}{h}$. 
\item $\delta_i=\partial_{x_i}$ and $\omega_i=dx_i$ for $i=k+1,\ldots,n$.
\end{enumerate}

\begin{lemma}[Curve Test]\label{lemma:curvetest}
Let $Z$ be an irreducible subvariety of $M$, let $p\in Z\cap D$ be any point and let $\Delta$ be a small disc in $\mathbb{C}$ centered at $0$ with radius $\epsilon$. Assume that for every complex analytic path germ $\phi\colon (\Delta, 0)\to (Z,p)$ such that $\phi(\Delta\setminus 0)\subset (Z\setminus D)$,  
the differential map $d\phi$ induces a morphism of logarithmic $1$-forms 
\[
\phi^*\colon \Omega^1_{M, p}(\log D)\to \Omega^1_{\Delta, 0}(\log \{0\})\/.
\] 
Then the subvariety $Z$ is weakly log transverse  to $D$ at $x$. 
\end{lemma}

\begin{proof}
Let $\Lambda=T^*_ZM$ and we prove by contradiction. Suppose $Z$ is not weakly log transverse to $D$ at $p\in Z\cap D$, then therer exists $s\neq 0$ and $\xi\in T^*_pM$ such that $j'_p(\xi,s)=\eta \in \Lambda^{log}$. Since $j'_p$ is linear and $\Lambda^{log}$ is a cone, we can assume $s=1$ for simplicity. Our choice of basis for $\Omega^1_{M,p}(\log D)$ implies that $\delta_1(\frac{dh}{h})=1$ and  $\delta_i(\frac{dh}{h})=0$ for $i=2,\ldots,k$. By \eqref{equation:j'} we have
\[
\eta=j'_p(\xi_1,\ldots,\xi_n,1)=\omega_1+\sum^n_{i=k+1} \xi_i\omega_i.
\]
Let $\pi$ be the natural projection $T^*M(\log D)\to M$. Choose a curve $\varphi:(\Delta,0) \to (\Lambda^{log},(p,\eta))$ such that $\pi\circ \varphi(\Delta\setminus 0)\subset (Z_{reg}\setminus D)$. Denote $\pi\circ \varphi$ by $\phi$. The map $\varphi$ can be written as
\[
\varphi(t)=\varphi_1(t)\omega_1+\ldots + \varphi_n(t)\omega_n,
\]
satifiying the properties 
\begin{enumerate}[label=(\roman*)]
\item $\lim_{t\to 0}\varphi_1(t)=1$, 
\item $\lim_{t\to 0}\varphi_i(t)=0$ for $i=2,\ldots,k$,
\item $\lim_{t\to 0}\varphi_i(t)=\xi_i$ for $i=k+1,\ldots,n$.
\end{enumerate}
Since the map $j$ induces an isomorphism between $\Lambda\vert_{U}$ and $\Lambda^{log}\vert_{U}$, for $t\neq 0$ the form $\varphi(t)$ can be regarded as a conormal vector to $Z$ expanded in terms of the logarithmic basis at $p$. Since $\phi(\Delta\setminus 0) \subset Z_{reg}$, the form $\varphi(t)$ should be conormal to $\phi(\Delta\setminus 0)$ as well. Pulling back to $\Delta$ we see that
\[
\phi^*(\varphi(t))=\varphi_1(t)\frac{dh(t)}{h(t)}+\varphi_2(t)\phi^*\omega_2+\ldots+\varphi_n(t)\phi^*\omega_n
\]
must be identically zero on $\Delta\setminus 0$. However, the assumption and the limits of $\varphi_i(t)$ imply that $\varphi_i(t)\phi^*\omega_i$ is holomorphic for any $i=2,\ldots,n$ so that
\[
\textup{ord}_t\Big(\phi^*(\varphi(t))\Big)=\textup{ord}_t\Big(\frac{dh(t)}{h(t)}\Big)=-1.
\]
This leads to a contradiction.
\end{proof}

\begin{example}\label{coro; NCimpliesWLT}
As an application of the curve test lemma, we can easily show the weak transversality of $Z$ and $D$ in the context of Theorem \ref{theo; NCimpliesLT}. Let $p\in Z\cap D$ be any point. Take an analytic neighbourhood $\mathcal{U}_p$ of $p$ such that the equation of $\mathcal{U}_p\cap D$ is given by $x_1\cdots x_k=0$. We take $\{\frac{dx_1}{x_1}, \ldots , \frac{dx_k}{x_k},dx_{k+1},\ldots,dx_n \}$ as a basis for $\Omega^1_{M, p}(\log D)$.  Let $\phi: (\Delta,0)\to (Z,p)$ be any analytic path germ such that $\phi(\Delta\setminus 0)\subset Z\setminus D$. Since $\phi^*(\frac{dx_i}{x_i})=\frac{d(x_i(t))}{x_i(t)}$ for $i=1,\ldots,k$ and $\phi^*(dx_j)=d(x_j(t))$ for $j=k+1,\ldots,n$, the pole orders of these forms are no greater than $1$. By Lemma \ref{lemma:curvetest}, $Z$ is weakly log transverse to $D$ at $p$. Since this holds for any $p\in Z\cap D$, $Z$ is weakly log transverse to $D$.
\end{example}

\begin{remark}
Lemma \ref{lemma:curvetest} provides an effective method to test weak log transversality of a curve to $D$, but when $\dim Z>1$ the condition in the lemma is too strong. For example, $M$ is always weakly log transverse to a strongly Euler homogeneous free divisor $D$, but there can be plenty of curves in $M$ not satisfying the curve test. See Example \ref{counterexample}.
\end{remark}

\begin{conj}
\label{conj; SNC}
Let $D$ be a reduced free divisor of $M$. If any irreducible subvariety $Z$ such that $Z\not\subset D$ is weakly log transverse to $D$, then $D$ is normal crossing.
\end{conj}

\section{Restriction of Hypersurface Complement to Free Divisor Complement}
\label{sec; doubleres}
Let $L$ and $L'$ be two line bundles  on $M$.
Suppose  we  are  given a  reduced hypersurface $V=V(g)$ and  a free divisor $D=V(f)$ cut by global sections $g\in H^0(M, L)$ and $f\in H^0(M, L')$.  
We now adapt our  global sharp constructions to investigate the restriction of characteristic cycles to the complement $M\setminus D\cup V$. 
Our strategy is to  restrict the characteristic cycle to $M\setminus V$ first using Ginzburg's sharp operation, and then restrict to $M\setminus D$ using logarithmic completion.

Similar to the definition of $\mathcal{P}^1_M\mathcal{L}$ in \S\ref{sec; sharp}, we consider the $\mathbb{C}_M$-module $\mathcal{L}\oplus (\mathcal{L}\otimes\Omega^1_M(\log D))$ endowed with a $\mathscr{O}_M$-module structure given by
\[
h\cdot (s,\alpha)= (hs, s\otimes dh + h\alpha).
\]
Since $D$ is assumed to be a free divisor, we obtain a locally free sheaf of rank $n+1$. We denote the associated vector bundle by $P^1_ML(\log D)$ and call it the \textbf{bundle of logarithmic principal parts of $L$ along $D$}. The natural inclusion $\Omega^1_M \to\Omega^1_M(\log D)$ induces the Cartesian diagram with exact rows below 

\begin{equation}
\begin{tikzcd}
\label{diagram; P^1LlogD}
0 \arrow{r} & T^*M\otimes L \arrow["i_1"]{r}\arrow["j_L"]{d} &  P^1_ML \arrow["u"]{d}   \arrow["q_1"]{r} & L \arrow{r}\arrow{d} & 0\\
0 \arrow{r} & T^*M(\log D)\otimes L \arrow["i_2"]{r} & P^1_ML(\log D)\arrow["q_2"]{r} & L \arrow{r} & 0
\end{tikzcd} \/.
\end{equation}  

The notation looks similar to some notations introduced in \S\ref{sec; sharp}, and the reader should keep in mind the difference in their meanings. In \S\ref{sec; sharp}, $\mathscr{O}(D)\cong L$ but here $\mathscr{O}(D)\cong L'$ which is not assumed to have any relation to $L$.  

\begin{remark}
\label{logsplitting}
In the context of Lemma~\ref{lemma; logconnection}, we have a splitting of $\mathcal{O}_M$-modules
\[
\mathcal{P}^1_M\mathcal{L}(\log D) \cong \mathcal{L} \oplus (\mathcal{L}\otimes\Omega^1_M(\log D)),
\]
since by the definition of the $\mathscr{O}_M$-structure of $\mathcal{P}^1_M\mathcal{L}(\log D)$, a splitting $\mathcal{L} \to \mathcal{P}^1_M\mathcal{L}(\log D)$ is equivalent to a logarithmic connection $\mathcal{L} \to \mathcal{L}\otimes\Omega^1_M(\log D)$.

When $L\not\cong L'$ we can't expect a splitting as above exists. 
However, we note that the splitting does not necessarily imply $L\cong L'$, since such splitting always exist when $D$ is a projective hyperplane arrangement  (see Proposition~\ref{hyparrlogjet}).  
\end{remark}

Let $Z\not\subset V$ be an irreducible subvariety of $M$ and let $\Lambda$ be its conormal space. 

\begin{definition} 
\label{defn:JsharplogD}
We define the $(\log D)$-jet sharp construction of $\Lambda$ to be
\[
J\Lambda^{\sharp,log}_{D}:= \overline{u\left(\overline{J\Lambda^{\sharp}_g}\right)}\subset P^1_ML(\log D) \/.
\]
In other words, we first perform the global jet sharp operation using $g\in H^0(M,L)$, and then perform the logarithmic completion along the free divisor $D=V(h)$. 
\end{definition}
We will use the shorthand notation $J\Lambda^{\sharp,log}$ for $J\Lambda^{\sharp,log}_{D}$ in the rest of this section.

We denote by $\tau_1\colon \overline{J\Lambda^\sharp}\to M$ and $\tau_2\colon J\Lambda^{\sharp,log}\to M$  the natural projections. By Proposition \ref{global-Ginzburg} we have $[i_1^{-1}(\overline{J\Lambda^{\sharp}})]=\textup{JCC}(\textup{Eu}_Z^\vee\cdot \ind_{M\setminus V})$  in  $T^*M\otimes L$. It's intriguing to ask what if we pullback $[J\Lambda^{\sharp,log}]$ along $i_2\circ j_L$. The following lemma answers this question when $D$ is  holonomic.

\begin{lemma}
\label{lemma; i2log}  
Assume that $D$ is a holonomic free divisor.
\begin{enumerate}[label=(\roman*)]
\item  
If $\textup{JCC}(\textup{Eu}_Z^\vee\cdot \ind_{M\setminus V})=\sum_\alpha m_\alpha [T_{Z_\alpha}^*M\otimes L]=\sum_\alpha m_\alpha [J\Lambda_\alpha]$, then
\[
\left[i_2^{-1}\left(J\Lambda^{\sharp,log}\right)\right]=
\sum_{Z_\alpha\not\subset D} 
m_\alpha \left[J\Lambda_\alpha^{log}\right] \/.
\]
In other words, $\left[i_2^{-1}\left(J\Lambda^{\sharp,log}\right)\right]$ is the log completion of  $\textup{JCC}(\textup{Eu}_Z^\vee\cdot \ind_{M\setminus V})$ along $D$.
\item If those $Z_\alpha$ such that $Z_\alpha\not\subset D$ are log transverse to $D$, we have:
\[  
\textup{JCC}(\textup{Eu}_Z^\vee\cdot \ind_{M\setminus D\cup V})=[(i_2\circ j_L)^{-1}(J\Lambda^{\sharp,log})] \/.
\]
\end{enumerate}
\end{lemma} 
\begin{proof}
We have to show that  $i_2^{-1}(J\Lambda^{\sharp,log})$ is the logarithmic completion  of the irreducible components of $i_1^{-1}(\overline{J\Lambda^{\sharp}})$ which are not supported on $D$. It suffices to check locally. 

In local coordinates, $i_1^{-1}(\overline{J\Lambda^{\sharp}})$ and  $i_2^{-1}(J\Lambda^{\sharp,log})$ are both cut by $s=0$ from $\overline{J\Lambda^{\sharp}}$ and $J\Lambda^{\sharp,log}$ respectively. We conclude that $i_2^{-1}(J\Lambda^{\sharp,log})=\overline{u(i_1^{-1}(\overline{J\Lambda^{\sharp}}))}$. Therefore, every irreducible component of $i_2^{-1}(J\Lambda^{\sharp,log})$ is the image closure of some irreducible component of $i_1^{-1}(\overline{J\Lambda^{\sharp}})$.

Since $u$ is an isomorphism on
$\overline{J\Lambda^{\sharp}}\setminus \tau_1^{-1}(D)\to J\Lambda^{\sharp,log} \setminus \tau_2^{-1}(D)$, there is a bijection between irreducible components of $i_1^{-1}(\overline{J\Lambda^{\sharp}})$ which are not contained in $\tau_1^{-1}(D)$ and irreducible components of $i_2^{-1}(J\Lambda^{\sharp,log})$ which are not contained in $\tau_2^{-1}(D)$. The multiplicities of the corresponding irreducible components are also the same. This shows that $[i_2^{-1}(J\Lambda^{\sharp,log})]$ contains the logarithmic completions of irreducible components of $i_1^{-1}(\overline{J\Lambda^{\sharp}})$ which are not supported on $D$ with desired multiplicity. 

To finish the proof, we still need to show that $i_2^{-1}(J\Lambda^{\sharp,log})$ has no undesirable irreducible components supported on $D$. Since both $i_1^{-1}(\overline{J\Lambda^{\sharp}})$ and $i_2^{-1}(J\Lambda^{\sharp,log})$ are equidimensional of the same dimension $n$, it suffices to show that, if $T^*_WM \otimes L$ is an irreducible component  of $i_1^{-1}(\overline{J\Lambda^{\sharp}})$ such that $W \subset D$, then $\dim j_L(T^*_WM \otimes L)<n$. It suffices to show that $j:T^*M\to T^*M(\log D)$ kills some non-zero conormal vector on generic points of $W$.

Consider the logarithmic stratification of the divisor $D$. 
For any $x\in W_{reg}$ we denote by $D_x$ the logarithmic stratum containing $x$. 
Since the logarithmic stratification is locally finite by our assumption and $W\subset D$, $D_x\cap W_{reg}$ is a neighborhood of $x$ in $W$ for generic $x\in W$. So any covector conormal to $D_x$ at $x$ is automatically conormal to $W$ at generic $x$. Since $j\vert_{D_x}$ can be identified with the natural restriction $T^*M\vert_{D_x} \to T^*D_x$ with kernel $T^*_{D_x}M$, our conclusion follows. 
 
The proof of the first statement is now complete. The second statement follows directly from the definition of log transversality. 
\end{proof}

Next we focus on the special case where $\gamma=\textup{Eu}^\vee_M=(-1)^n \ind_M$, $V$ is a reduced hypersurface and $D$ is a holonomic free divisor. We explained in the proof of Proposition \ref{prop; lineartype} that
$$\overline{J\Lambda^\sharp}=\textup{Spec} \left( \textup{Rees}_{\mathscr{O}_M}  \mathcal{J}_V \right) 
 =\text{Spec} \left( \mathscr{O}_M\oplus  \mathcal{J}_Vt \oplus \mathcal{J}^2_Vt^2 \ldots \right)  
 \subset P_M^1L $$ 
where $\mathcal{J}_V$ is the image of the morphism $(\mathcal{P}^1_M\mathcal{L})^\vee \to \mathscr{O}_M$ induced by $j(g)$.

The map $u:P^1_ML \to P^1_ML(\log D)$ corresponds to a morphism of locally free sheaves 
\[
u^\vee:(\mathcal{P}^1_M\mathcal{L}(\log D))^\vee \to (\mathcal{P}^1_M\mathcal{L})^\vee.
\]

\begin{definition}
\label{defn; J_VlogD}
We define $\mathcal{J}_V(\log D)$, the ideal sheaf of logarithmic Jacobian of $V$ with respect to $D$, as the image of the composition
\[
(\mathcal{P}^1_M\mathcal{L}(\log D))^\vee \xrightarrow{u^\vee} (\mathcal{P}^1_M\mathcal{L})^\vee \xrightarrow{j(g)} \mathscr{O}_M.
\]
\end{definition} 

Let $\{U_\alpha\}$ be an open cover of $M$ such that $L$, $T^*M$ and $T^*M(\log D)$ are trivialised. If $g_\alpha$ is a local equation of $g$ on an open subset $\mathcal{U}_\alpha\subset M$, and $\{\delta_1, \cdots ,\delta_n\}$ is a basis for $\textup{Der}_M(-\log D)(\mathcal{U}_\alpha)$, then the ideal $\mathcal{J}_V(\log D)(\mathcal{U}_\alpha)$ is $(\delta_1(g_\alpha), \cdots ,\delta_n(g_\alpha), g_\alpha)$.

\begin{proposition}
Let $\Lambda=M=T^*_MM$. We have
\[
J\Lambda^{\sharp,log}=\textup{Spec} \Big( \textup{Rees}_{\mathscr{O}_M}  \mathcal{J}_V(\log D) \Big) 
 \subset P_M^1L(\log D) \/.
\]
\end{proposition}

\begin{proof}
The proof follows easily from Definition \ref{defn:JsharplogD} and Definition \ref{defn; J_VlogD}.
\end{proof}

We will compute $c_*(\ind_{V\setminus D})$ and $c_*(\ind_{U})$ with the help of this construction.
\begin{theorem}
\label{theo; csmdoubleres}
If $D$ is a holonomic strongly Euler homogoneous free divisor and $\ind_V$ is log transverse to $D$, then
\begin{align*}
&c_{*}(\ind_{V\setminus D})= 
c(TM(-\log D))c(L)^{-1}\cap 
\Big(
[V]+s\left(\mathcal{J}_V(\log D),M\right)^\vee\otimes_M L
\Big),  \\
&c_{*}(\ind_{M\setminus(V\cup D)}) = c(TM(-\log D))c(L)^{-1}\cap \Big([M]-s\left(\mathcal{J}_V(\log D), M\right)^\vee\otimes_M L \Big).
\end{align*}
Here $s\left(\mathcal{J}_V(\log D), M\right)^\vee$ and $\otimes_{M}$ use the dual and tensor notation introduced in \cite{Aluffi03}. 
\end{theorem}

Let $\tau_1: \overline{J\Lambda^\sharp}  \to M$ and $\tau_2: J\Lambda^{\sharp,log}  \to M$ be the natural projections. 
The map $u$ restricts to $u: \overline{J\Lambda^\sharp}  \to J\Lambda^{\sharp,log} $ commuting with projections $\tau_1, \tau_2$. We denote by $\textup{Sing}(V,\log D)$ the subscheme of $M$ defined by $\mathcal{J}_V(\log D)$. Theorem \ref{theo; csmdoubleres} will be proved after the following preparatory lemma.

\begin{lemma} \ 
\label{lemma; hyper-log} 
\begin{enumerate}[label=(\roman*)]
\item $[i_1^{-1}(\overline{J\Lambda^\sharp} )]=\textup{JCC}(\ind_{M\setminus V}\cdot\textup{Eu}^\vee_{M}) = [M]+[\tau_1^{-1}(V)]-[\tau_1^{-1}\textup{Sing}(V)].$
\item  
$ 
[i_2^{-1}(J\Lambda^{\sharp,log} )] =[M] + [\tau_2^{-1}(V)] - [\tau_2^{-1}\textup{Sing}(V,\log D)]$
\end{enumerate}
\end{lemma}
\begin{proof}
First we prove the formula (i). 
The first equality is due to the global Ginzburg's contruction. The second equality is the reformulation of a well-known result of Aluffi \cite{MR1697199} and Parusinski-Pragacz \cite{MR1795550} in terms of the jet conormal space. Here we give a simple and illuminating argument based on global Ginzburg's construction. On $U_\alpha$, $\overline{J\Lambda^\sharp}$ is given by the Rees algebra
\begin{equation*}
\textup{Rees}(J) = \mathscr{O}_{U_\alpha} \oplus J t \oplus J ^2t^2\oplus \ldots
\end{equation*}
where $J$ is the ideal  generated by $g_\alpha$ and all its first partial derivatives. We have 
\[
\textup{Spec}(\textup{Rees}(J))=\overline{J\Lambda^\sharp}\vert_{U_\alpha} \cong \overline{(T^*_{U_\alpha}U_\alpha)^{\sharp}_{g_\alpha}}
\]
according to Proposition \ref{prop:trivialisation of JLambda}, where we also showed the function $s$ on the latter corresponds to the function $g_\alpha t$ on the former. Since Proposition \ref{CC-G} says that $\textup{CC}(\textup{Eu}^\vee_{U_\alpha}\cdot \ind_{U_\alpha \setminus V})$ is obtained by settting $s=0$ in the latter, we must set $g_\alpha t=0$ in $\textup{Spec}(\textup{Rees}(J))$ to obtain the jet characteristic cycle
\[
\textup{JCC}(\textup{Eu}^\vee_M \cdot \ind_{M\setminus V})\vert_{U_\alpha}=[i_1^{-1}\overline{J\Lambda^{\sharp}}]\vert_{U_\alpha}.
\] 
In the Rees algebra, $g_\alpha$ defines $\tau_1^{-1}(U_\alpha \cap V)$, $J$ defines the normal cone $\tau_1^{-1}(\textup{Sing}(U_\alpha \cap V))$ and $Jt$ defines the zero section $U_\alpha$, reading the equality of ideals $(g_\alpha t)=(g_\alpha \cdot Jt):J$ schematically, we have
\begin{equation*}
[i_1^{-1}(\overline{J\Lambda^\sharp} )]\vert_{U_\alpha} = [\tau_1^{-1}(V\cap U_\alpha)] + [U_\alpha] - [\tau_1^{-1}\textup{Sing}(V\cap U_\alpha)]. 
\end{equation*}
Since $\{U_\alpha\}$ is an open cover of $M$, formula (i) follows. Formula (ii) can be proved analogously. We leave the details to the reader.   
\end{proof}

\begin{proof}[Proof of Theorem \ref{theo; csmdoubleres}] 
Let $B=\mathbf{P}\left(J\Lambda^{\sharp,log}\right)$. Denote the natural projection $B\to M$ by $p$ and denote the tautological subbundle  of $\mathbf{P}(P^1_{M}L(\log D))$ by $\xi\otimes L$. Hence the tautological bundle of $\mathbf{P}(T^*M(\log D))$ can be written as $\xi$. By Lemma \ref{lemma; i2log} and Lemma \ref{lemma; hyper-log}, $[M] + [\tau_2^{-1}(V)] - [\tau_2^{-1}\textup{Sing}(V,\log D)]$ is the log completion of $[M]+[\tau_1^{-1}(V)]-[\tau_1^{-1}\textup{Sing}(V)]$. So $[\tau_2^{-1}(V)] - [\tau_2^{-1}\textup{Sing}(V,\log D)]$ is the log completion of $[\tau_1^{-1}(V)]-[\tau_1^{-1}\textup{Sing}(V)]$. Note that 
\[
[\tau_1^{-1}(V)]-[\tau_1^{-1}\textup{Sing}(V)]=\textup{JCC}(\textup{Eu}^\vee_M \cdot \ind_{M\setminus V}-\textup{Eu}^\vee_M)=\textup{JCC}((-1)^{n+1}\cdot \ind_V).
\]

By Theorem \ref{theo;logCCtochern} we have
\begin{equation}\label{equation:chernVminusD}
c_*(\ind_{V\setminus D})=c(TM(-\log  D))\cap p_*\bigg(c(\xi^\vee)^{-1}\Big(c_1(L)-c_1(\xi \otimes L)\Big)\cap [B]\bigg)
\end{equation} 

To simplify this expression, let $\mathcal{X}=c_1(L)$ and $\mathcal{Y}=c_1(\xi\otimes L)$. We have 
\begin{align*}
c(\xi^\vee)^{-1}\Big(c_1(L)-c_1(\xi \otimes L)\Big)&=\frac{\mathcal{X}-\mathcal{Y}}{1+\mathcal{X}-\mathcal{Y}}  
 =\frac{\mathcal{X}}{1+\mathcal{X}}-\frac{\mathcal{Y}}{(1+\mathcal{X})(1+\mathcal{X}-\mathcal{Y})} \\
&=\frac{\mathcal{X}}{1+\mathcal{X}} - \frac{\mathcal{Y}}{(1+\mathcal{X})^2}\frac{1}{1-\frac{\mathcal{Y}}{1+\mathcal{X}}} \\
&=\frac{\mathcal{X}}{1+\mathcal{X}} - \frac{1}{1+\mathcal{X}}\Big(\frac{\mathcal{Y}}{1+\mathcal{X}}+\frac{\mathcal{Y}^2}{(1+\mathcal{X})^2}+\ldots\Big).
\end{align*}

Capping this equation with $[B]$ and applying $p_*$, we have
\begin{align*}
&p_* \bigg(c(\xi^\vee)^{-1}\Big(c_1(L)-c_1(\xi \otimes L)\Big) \cap [B]\bigg) \\
=& c(L)^{-1}\cap \Big([V] - \frac{s_{n-1}}{c(L)}+\frac{s_{n-2}}{c(L)^2}-\ldots \Big)
\end{align*}
where $p_*(\mathcal{Y}^{i}\cap [B]) = (-1)^{i-1}s_{n-i}$ and $s=s(\mathcal{J}_V(\log D), M)=s_0 + s_1 +\ldots$ is the Segre class of the cone $\tau_2^{-1}(\textup{Sing}(V,\log D))$. Reindexing the Segre class by codimension, i.e. $s_{n-i}=s^i$, using the dual and tensor notation in \cite{Aluffi03} we have
\begin{align*}
&p_* \bigg(c(\xi^\vee)^{-1}\Big(c_1(L)-c_1(\xi \otimes L)\Big) \cap [B]\bigg) \\
=& c(L)^{-1}\cap \Big([V] - \frac{s^1}{c(L)}+\frac{s^2}{c(L)^2}-\ldots \Big) \\
=& c(L)^{-1}\cap \Big([V]+s^\vee \otimes_{M} c(L)\Big) \/.
\end{align*}

Combining with eqution \eqref{equation:chernVminusD} we obtain the formula for $c_*(\ind_{V\setminus D})$. 

By Corollary \ref{coro; equiSEH} we have $c_*(\ind_{M\setminus D})=c(TM(-\log  D))\cap [M]$, therefore
\begin{equation*}
c_*(\ind_{M\setminus(V\cup D)}) = c_*(\ind_{M\setminus D})-c_*(\ind_{V\setminus D})=c(TM(-\log D))c(L)^{-1}\cap \Big([M]-s^\vee\otimes_{M} c(L) \Big).
\end{equation*}
\end{proof}

\section{Non-negativity of Euler Characteristics}
\label{sec; nonnegative}
As an application of  previous constructions  we discuss some non-negative results concerning the Euler characteristics of constructible functions. To start we recall the following key proposition  from \cite[Proposition 2.3]{DPS94}:
\begin{proposition} 
\label{prop; nefimplynonnegative}
Let $M$ be a $n$-dimensional projective manifold and  $E$ be a  holomorphic vector bundle of rank $n$. 
If $E$ is nef, then for any $n$-dimensional cone $C$ in $E$  we have 
\[
\deg \left( [C]\cdot [M] \right) \geq 0 \/.
\]
Here  $M$ denotes the zero section of $E$.
\end{proposition} 
We refer to \cite[\S 6]{MR2095472} for the precise definition and basic properties of ample and  nef vector bundles.  Here we mainly focus on the following examples.
\begin{example}\ 
\begin{enumerate} 
\item Globally generated vector bundles are nef.
\item Extension, quotient and tensor product of nef vector bundles are nef.  
\item Given any $E$, $E\otimes L^d$ is nef for any ample line bundle $L$ with $d$ sufficiently large. 
\end{enumerate}  
\end{example}

\begin{definition} Let  $\gamma=\sum_{\alpha} n_\alpha \textup{Eu}^\vee_{Z_\alpha}$ be a constructible function on $M$.
\begin{enumerate}
\item We say $\gamma$  is effective if $n_\alpha\geq 0$ for all $\alpha$.
\item  Given  any constructible subset $Y$, we define  the Euler characteristic of $\gamma$ on $Y$ by
$\chi(Y, \gamma):= \chi(M, \gamma\cdot \ind_Y)= \sum_{\beta} m_\beta\chi(S_\beta)$
providing that $\gamma\cdot \ind_Y=\sum_{\beta} m_\beta \ind_{S_\beta}$.
\end{enumerate}
\end{definition}

\begin{remark} 
The constructible functions assigned to  perverse sheaves are effective, but not every effective constructible function comes from a perverse sheaf. The effective constructible functions are more general in the discussion of Euler characteristics.
\end{remark}

An immediate consequence of Proposition~\ref{prop; genericZ} and Theorem~\ref{theo; NCimpliesLT} is the following.
\begin{corollary}
\label{coro; perverseisnonnegative}
Assume   $M$ is a projective manifold, $D$ is a strongly Euler homogeneous  free divisor of $M$ and  $T^*M(\log D)$ is nef.   Let $\gamma$ be an effective constructible function on $M$.
If either the logarithmic strata of $D$ are non-characteristic with respect to  
$\textup{CC}(\gamma)$, or  $D$ is normal crossing, 
then   $\chi(M\setminus D, \gamma)\geq 0 $.
\end{corollary}
\begin{proof}  
From Proposition~\ref{prop; genericZ} and Theorem~\ref{theo; NCimpliesLT} we know that in both cases $\gamma$ is weakly log transverse to $D$.  Thus Equation~\eqref{relation; CCnonequivariant} applies and we have, setting $U=M\setminus D$,
\[
\chi(U, \gamma)=\chi(M, \gamma \cdot \ind_U)= \deg \left( [\textup{CC}(\gamma)^{\log}]\cdot [M] \right) \/.
\]
As $\gamma$ is effective,  
$\textup{CC}(\gamma)^{\log}$ is an effective $n$-dimensional conical cycle in $T^*M(\log D)$. Since $T^*M(\log D)$ is nef, the intersection number is non-negative by Proposition~\ref{prop; nefimplynonnegative}.
\end{proof}

\begin{example} 
Assume that $D\subset \mathbb{P}^n$ is  a  strongly Euler homogeneous linear free divisor. 
Then  $T^*\mathbb{P}^n(\log D)$ is  trivial and hence nef.   
Thus an effective constructible function on $\mathbb{P}^n$  has non-negative Euler characteristic if the closure of its support is  in  general position. 

If $Z\subset \mathbb{P}^n\setminus D$ is an algebraic  complete intersection subvariety, then  $\mathbb{C}_Z[\dim Z]$ is perverse and  $(-1)^{\dim Z}\ind_Z$ is effective. Then if $\overline{Z}$ is in  general position, we have $(-1)^{\dim Z}\chi(Z) \geq 0$.

The coordinate  arrangement $D=V(x_0\cdots x_n)$  is both linear free  and normal crossing, and the complement $U=(\mathbb{C}^*)^n$ is   the   affine torus.  Then any effective algebraic constructible function on $(\mathbb{C}^*)^n$  has non-negative Euler characteristic.  Similarly the Euler characteristics of  effective algebraic constructible functions on semi-abelian varieties are non-negative.
This was first proved in \cite[Corollary 1.2]{MR1769729}.
\end{example}

Proposition~\ref{prop; nefimplynonnegative} may fail if the log cotangent bundle is not  nef.   
The nefness can be settled by tensoring with  ample line bundles, but the intersection number in the twisted vector bundle changes. The following proposition addresses this twisted intersection number.

Let $L$ be a very ample line bundle on the projective  manifold $M$.  
Recall from Definition~\ref{defn;jetconormal} that the jet characteristic cycle   is a $n$-dimensional conical cycle in $T^*M\otimes L$. 
\begin{proposition}
\label{prop; Jchi} 
Denote by $M$ the zero  section of $T^*M\otimes L$. Given any  constructible function $\gamma$ on $M$, for a generic hypersurface $V$ in the linear system $|L|$ we have:
\[
\deg \left([\textup{JCC}(\gamma)]\cdot [M] \right) =   \chi(M\setminus V, \gamma) \/.
\]
\end{proposition}

\begin{proof}
It suffices to prove for $\gamma=\textup{Eu}^\vee_Z$, where $Z\subset M$ is an irreducible closed subvariety.  We denote by $U$ the  complement $M\setminus V$ 
and by $\breve{c}_*(\textup{Eu}_Z^\vee)=\sum_i (-1)^{n-i}\alpha^i$ the signed Chern-Mather class of $Z$. Here $\alpha^i\in H_{n-i}(M)$ denotes the  codimension  $i$ piece. 
Since $\breve{c}_*(\textup{Eu}_Z^\vee)$ is the shadow of $CC(\gamma)=T_Z^*M$ in $T^*M$, by \cite[Lemma 2.1]{AMSS23} we have 
\[
 [\textup{JCC}(\textup{Eu}_Z^\vee)]\cdot [M] = \sum_{i=0}^n  (-1)^{n-i} [V]^{n-i}\alpha^i   \/.  
\]
Let $\nu\colon M\to \mathbb{P}^N$ be the embedding induced by $L$. Then $[V]=\nu^*[\mathcal{H}]$ for a generic hyperplane $\mathcal{H}$ of $\mathbb{P}^N$. 
By \cite[Proposition 2.6]{Aluffi13} we have 
\begin{align*}
\chi(U, \textup{Eu}^\vee_Z\cdot \ind_U) =&
\int_{M} c_*(\textup{Eu}^\vee_Z\cdot \ind_U) = \int_{\mathbb{P}^N} \nu_*c_*(\textup{Eu}^\vee_Z\cdot \ind_U)\\
=& 
\int_{\mathbb{P}^N}  
c_*(\textup{Eu}^\vee_Z \cdot \ind_{\mathbb{P}^N\setminus \mathcal{H}}) 
=  \int_{\mathbb{P}^N} \frac{\nu_*c_*(\textup{Eu}^\vee_Z)}{1+[\mathcal{H}]}\\
=& \int_{M} \frac{c_*(\textup{Eu}^\vee_Z)}{1+[V]}
=\deg \left(  \sum_{i=0}^n  (-[V])^{n-i} \alpha^i \right)  \/.    
\end{align*}   
\end{proof}

\begin{example}
When $M=\mathbb{P}^n$, for $d\geq 2$ the twisted cotangent bundle $T^*\mathbb{P}^n\otimes \mathscr{O}_{\mathbb{P}^n}(d)$ is globally generated and hence nef. Thus given any effective constructible function $\gamma$, for a generic degree $d\geq 2$ hypersurface $V$ we have 
$\chi(\mathbb{P}^n\setminus V, \gamma)\geq 0$.
\end{example}

Motivating from  \cite{MRWW24} we  prove the following non-negative result. 
 Let $M$ be a $n$-dimensional projective   manifold  and  $D\subset M$ be a holonomic strongly Euler homogeneous free divisor.  
\begin{theorem}
\label{theo; nonnegativeperversetriple}
Given an effective constructible function   $\gamma$ on $M$.
Let $V$ be a reduced hypersurface of $M$ such that $L:=\mathscr{O}_M(V)$ is very ample and $T^*M(\log D)\otimes L$ is nef. 
Let $V'$ be a generic  hypersurface in the linear system $|L|$.  
\begin{enumerate}
\item If $\gamma$  is weakly log transverse to $D$,    we have  
\[
 \chi(M\setminus (D\cup V'), \mathcal{F}^\bullet)\geq 0 \/.
\]
\item If  $\gamma\cdot \ind_{M\setminus V}$ is weakly log transverse to $D$,   we have  
\[
 \chi(M\setminus (D\cup V\cup V'), \mathcal{F}^\bullet)\geq 0 \/.
\]
\end{enumerate}
\end{theorem}
\begin{proof}
Let $j\colon  M\setminus D\to M$ be the open immersion. 
By assumption $T^*M(\log D)\otimes L$ is nef, then $P_{M}^1L(\log D)$ is also nef since it is an extension of nef vector bundles. By Proposition~\ref{prop; nefimplynonnegative}, the intersection number of $[\textup{JCC}(\gamma)^{\log}]$  with the zero section in $T^*M(\log D)\otimes L$
and  the intersection number of  $[\textup{JCC}(\gamma)^{\sharp, \log}]$  with the zero section in $P_{M}^1L(\log D)$ are both non-negative.  Then the theorem follows from Lemma~\ref{lemma; i2log} and Proposition~\ref{prop; Jchi}.
\end{proof}



\begin{corollary}
\label{corollary; Dktriple}
Let $D_k=V(x_0\cdots x_k)\subset \mathbb{P}^n$ be the union of $k$ coordinate hyperplanes. 
Given an  effective constructible function $\gamma$ on $\mathbb{P}^n$ and  a   reduced projective hypersurface $V$, for a generic hypersurface $V'$ in the linear system $|\mathscr{O}_{\mathbb{P}^n}(V)|$  we have
$$
 \chi(\mathbb{P}^n\setminus (D_k \cup V'), \gamma)\geq 0 \text{ and }  \chi(\mathbb{P}^n\setminus (D_k\cup V\cup V'), \gamma)\geq 0 \/.
$$
\end{corollary}
When $k=0$, $D=H_\infty$ is the hyperplane at infinity and $V'=H$ is a generic hyperplane, this recovers \cite[Corollary 1.2]{MRWW24} and hence the classical result that $\chi(\mathbb{C}^n\setminus H, \gamma)\geq 0$.  

\begin{proof}
The divisor $D_k$  is normal crossing  and  the weakly log transverse assumption is satisfied by Theorem~\ref{theo; NCimpliesLT}. Set $d=\deg V$,  by Proposition~\ref{prop; chernderlog} $T^*\mathbb{P}^n(\log D_k)\otimes \mathscr{O}(d)$ is nef whenever $d\geq 1$. Then Theorem~\ref{theo; nonnegativeperversetriple} applies.
\end{proof}

\begin{example}
Let $D=\mathcal{A}$ be a centrally free projective hyperplane arrangement. If all the exponents $d_i$ of $\mathcal{A}$  are $\geq 1$,   
$T^*\mathbb{P}^n( \log \mathcal{A})$  is   nef by Proposition~\ref{prop; chernderlog}.  Then by Corollary~\ref{coro; perverseisnonnegative} a  `general' effective constructible function on $\mathbb{P}^n$  has non-negative Euler characteristic. 

In general $T^*\mathbb{P}^n(\log \mathcal{A})\otimes \mathscr{O}(d)$ is nef whenever $d\geq 1$.  
Thus replacing $D_k$ by $\mathcal{A}$, under the  assumption of weakly log transverse Corollary~\ref{corollary; Dktriple} also holds.
\end{example}

\begin{remark}
Notice that all the divisor complements in $\mathbb{P}^n$  are affine.    
Based on this  the  following  result on perverse sheaves was proved in \cite{SSW25}  using generic vanishing. 

Given any two reduced hypersurfaces $V_1, V_2$ in $\mathbb{P}^n$ and any perverse sheaf $\mathcal{F}^\bullet$ on $\mathbb{P}^n\setminus V_1$. For a generic translate $gV_2$ of $V_2$ one has
$\chi(\mathbb{P}^n\setminus V_1\cup gV_2, \mathcal{F}^\bullet)\geq 0$.
\end{remark}

\section{Chern classes of the Complement of  Centrally Free Projective Hypersurfaces}
\label{generalisation-FMS}
In this section we apply earlier results to the case $M= \mathbb{P}^n$, $V=V(g)$ is a reduced hypersurface and $D$ is either a holonomic linear free divisor or a free projective hyperplane arrangement. 

\subsection{The Case of Linear Free Divisors}
\label{sec; linearfree}
First we discuss   a generalization of a recent result \cite[Theorem 1.2]{FMS} of  Fassarella, Medeiros and Salomao. 
Let $\hat{D}$ be the affine cone of $D$ and assume $\hat{D}$ is  linear free (see Definition~\ref{defn; linearfreedivisor}). Let $\mathcal{B}=\{\delta_0, \cdots ,\delta_n\}$ be a basis of $\textup{Der}_{\mathbb{C}^{n+1}}(-\log \hat{D})$ where each $\delta_i$ is homogeneous of polynomial degree $1$. The homogeneous ideal $(\delta_0(g),\ldots,\delta_n(g))$ of $\mathbb{C}[x_0,\ldots,x_n]$ generates the ideal sheaf $\mathcal{J}_V(\log D)$ (see Definition \ref{defn; J_VlogD}), and it also defines a rational map
\[
T_\mathcal{B}(g)\colon \mathbb{P}^n \dashrightarrow \mathbb{P}^n, \ p \mapsto [\delta_0(g)(p):\ldots : \delta_n(g)(p)]
\]
The multi-degrees (see Definition \ref{defn; multidegree}) of $T_{\mathcal{B}}(g)$ are independent of the choice of $\mathcal{B}$.
\begin{corollary}
\label{coro; linearfreecsm}
Let  $m_k(V,\log D)$ be the $k$th multi-degree of the rational map $T_\mathcal{B}(g)$.
Assume $\hat{D}$ is a holonomic linear free divisor and $\ind_V$ is log transverse to $D$. We have
\[
c_*(\ind_{\mathbb{P}^n\setminus V\cup D}) = \sum_{k=0}^n (-1)^k m_k(V,\log D)\cdot [\mathbb{P}^{n-k}] \in H_*(\mathbb{P}^n)\/.
\]
In particular, the degree of  $T_\mathcal{B}(g)$ equals $(-1)^n \chi(\mathbb{P}^n\setminus V\cup D)$. 
\end{corollary}
\begin{proof}
The projective divisor $D$ is holonomic by Proposition~\ref{prop:coneproperties}. We also have
\[
c(T\mathbb{P}^n(-\log D))=1 
\]
by Proposition \ref{prop; chernderlog}. From Theorem~\ref{theo; csmdoubleres} we obtain
\[
c_*(\ind_{\mathbb{P}^n\setminus V\cup D})=c(\mathscr{O}_{\mathbb{P}^n}(V))^{-1}\cap 
\Big(
[\mathbb{P}^n]-s\left(\mathcal{J}_V(\log D), \mathbb{P}^n\right)^\vee\otimes \mathscr{O}_{\mathbb{P}^n}(V)
\Big) \/.
\]
Applying Proposition~\ref{prop; muldegandsegre} we have 
\[
\sum_{k=0}^n (-1)^k m_k(V,\log D)\cdot [\mathbb{P}^{n-k}]=c(\mathscr{O}_{\mathbb{P}^n}(V))^{-1}\cap 
\Big(
[\mathbb{P}^n]-s\left(\mathcal{J}_V(\log D), \mathbb{P}^n\right)^\vee\otimes \mathscr{O}_{\mathbb{P}^n}(V)
\Big) \/.
\]
Combining the two equations, the corollary follows.
\end{proof}

\begin{example}  
\label{exam; FMS}
Let $D=V(x_0\cdots x_n)$ be the union of all coordinate hyperplanes in $\mathbb{P}^n$ and let $U=\mathbb{P}^n\setminus D$. The affine cone $\hat{D}$ is a normal crossing divisor hence is also linear free. For any hypersurface $V$, $\ind_V$ is log transverse to $D$ by Theorem \ref{theo; NCimpliesLT}. 

Choose $\mathcal{B}=\{x_0\partial_{x_0}, \ldots ,x_n\partial_{x_n}\}$ so that $\mathcal{J}_V(\log D)=(x_0g_{x_0}, \ldots, x_ng_{x_n})$. The map $T_{\mathcal{B}}(g)$ defined by the generators of $\mathcal{J}_V(\log D)$ is called the toric polar map of $g$ in the terminology of \cite{FMS}. In this case Corollary \ref{coro; linearfreecsm} is reduced to \cite[Theorem 1.2]{FMS}. 
\end{example}

When the affine cone $\hat{D}=V(f)$ is a linear free divisor we may also consider  the gradient map of $D$. Notice that $f$ is  necessarily homogeneous of degree $n+1$ and  $f_{x_i}$ is homogeneous of degree $n$ for each $i$. Denote by $\{g_i\}$ the set of multi-degrees of the gradient map
$$
\nabla_f\colon \mathbb{P}^n \dashrightarrow \mathbb{P}^n\/,  x\mapsto [f_{x_0}(x):\cdots :f_{x_n}(x)] \/.
$$
Let $\mathcal{J}_D=(f_{x_0}, \cdots, f_{x_n})$ be the ideal sheaf of $\textup{Sing}(D)$. By Proposition~\ref{prop; muldegandsegre} we have
\begin{equation}
\label{eq, aa}
\sum_{i=0}^n (-1)^i g_i [\mathbb{P}^{n-i}] = \frac{1}{1+nH}\cap \Big([\mathbb{P}^n]-s(\mathcal{J}_D, \mathbb{P}^n)^\vee\otimes \mathscr{O}(n)  \Big).
\end{equation} 

In practice it could  be hard  to check whether a linear free divisor is strongly Euler homogeneous, since it requires searching for local strongly Euler vector fields around every point. 
The following result provides a necessary numerical condition to check strongly Euler homogeneity through a global computation of multi-degrees or Segre classes, which can be effectively performed using   computer algebras.

\begin{corollary}\label{coro; numerical test for linearfree}
Let $\hat{D}=V(f)\subset \mathbb{C}^{n+1}$ be a strongly Euler homogeneous holonomic linear free divisor and let $\{g_i\}$  be the multi-degrees of the gradient map $\nabla_f$. We have 
\[
(1+nH)(1-H)^n \cap [\mathbb{P}^n]=[\mathbb{P}^n]-s(\mathcal{J}_D, \mathbb{P}^n)^\vee\otimes \mathscr{O}(n).
\]
This is equivalent to $g_i=\binom{n}{i}$ for  $i=0, \cdots ,n$.
\end{corollary}

\begin{proof} 
By Proposition~\ref{prop:coneproperties}, Corollary \ref{coro; chernclass} and Corollary \ref{coro; equiSEH} we have
\[
c_*(\ind_{\mathbb{P}^n\setminus D}) = c(T\mathbb{P}^n(-\log D))\cap [\mathbb{P}^n]=[\mathbb{P}^n] \/.
\]
Recall  from  \cite[Theorem 2.1]{Aluffi03} that
$$\sum_{i=0}^n g_i (-H)^i(1+H)^{n-i}\cap [\mathbb{P}^n]= c_*(\ind_{\mathbb{P}^n\setminus D})\/,$$ 
then we have
$
\sum_{i=0}^n g_i (-H)^i(1+H)^{n-i}=1 \/.
$ This forces $g_i=\binom{n}{i}$ for each $i$. The other statement follows from equation $(\ref{eq, aa})$.
\end{proof}

\begin{corollary}
\label{coro; linearfree}
Let $\hat{D}=V(f)\subset \mathbb{C}^{n+1}$ be a holonomic linear free divisor. Let $\{g_i\}$ be the multi-degrees of the gradient map $\nabla_f$. If $g_i\neq \binom{n}{i}$ for some $i$, then
the logarithmic comparison theorem does NOT hold for $\hat{D}$. 
\end{corollary}

\begin{proof}
This is a direct consequence of Corollary \ref{coro; equiSEH}, Corollary \ref{coro; numerical test for linearfree} and \cite[Theorem 4.7]{MR3366865}.
\end{proof}

\begin{example}[\cite{MR2521436}]
\label{exam; linearfree}
 Let $\hat{D}:=V(f)\subset \mathbb{C}^4$ be the reductive linear free divisor defined by
\[
f:=y^2z^2-4xz^3-4y^3 w+ 18xyzw - 27w^2x^2 \/.
\]
It was shown in \cite[Theorem 7.10]{MR2521436} that $\hat{D}$ is locally quasi-homogeneous, 
and hence is holonomic by \cite[Theorem 4.3]{CN02}.  Thus $\tilde{D}$ is a  holonomic strongly Euler homogeneous linear free divisor, since locally quasi-homogeneity clearly  implies  strongly Euler homogeneity. 

Let $D=\mathbf{P}(\hat{D})$. Following Proposition~\ref{coro; numerical test for linearfree} a direct computation shows that
\[
s(\mathcal{J}_D, \mathbb{P}^3)= -24[\mathbb{P}^0]+6[\mathbb{P}^1] \/.
\]
 
The module  $\textup{Der}_{\mathbb{C}^4}(-\log \hat{D})$ has a free basis $\mathcal{B}$ as follows:
\[
\begin{bmatrix}
\delta_0\\ \delta_1\\ \delta_2\\ \delta_3 
\end{bmatrix}=\begin{bmatrix}
x & y & z & w \\
3x &  y & -z & -3w \\
0 & 3x & 2y & z \\
y & 2z & 3w & 0
\end{bmatrix}\cdot\begin{bmatrix}
\partial_x \\\partial_y \\\partial_z\\\partial_w 
\end{bmatrix} \/.
\]

If we take a generic linear form $l=ax+by+cz+w$ that cuts a generic hyperplane $\mathcal{H}$, then $\mathcal{J}_\mathcal{H}(\log D)=(x,y,z,w)$. Since $\ind_\mathcal{H}$ is log transverse to $D$, we have
\[
\sum_{k=0}^3 (-1)^k m_k(\mathcal{H}, \log D)\cdot [\mathbb{P}^{3-k}]= c_*(\ind_{\mathbb{P}^3\setminus \mathcal{H}\cup D}) 
=\sum_{k=0}^3 (-1)^{3-k}[\mathbb{P}^{k}] \/.
\]
In other words, $m_k(\mathcal{H},\log D)=1$ for $k=0,1,2,3$. Another way to obtain these numbers is to observe that the rational map $T_{\mathcal{B}}(l)$ is an isomorphism. 

On the other hand, one may also view $\mathcal{H}$ as the free divisor and consider the logarithmic restriction of $D$ to $\mathbb{P}^3\setminus \mathcal{H}$. In this case the logarithmic ideal is
\[
\mathcal{J}_D(\log \mathcal{H})=(f,f_z-cf_w, f_y-bf_w, f_x-af_w ) \/.
\]
Compare with the previous calculation, this ideal is however  very complicated. Its Segre class and multi-degrees are also  difficult to compute. 

This example suggests that, when we restrict a characteristic cycle to the complement of the union of two free divisors, choosing which one to operate the log completion could be sensitive. 
\end{example}

\subsection{The Case of Hyperplane Arrangements}
\label{sec; hyperplanes}
In this subsection we let $D=\mathcal{A}$ be a projective hyperplane arrangement. We assume $\mathcal{A}$ is centrally free, i.e., its affine cone  $\hat{\mathcal{A}}$ is a free (central) hyperplane arrangement in $\mathbb{C}^{n+1}$. Since $\mathcal{A}$ is locally quasi-homogeneous and holonomic, Theorem \ref{theo; csmdoubleres} applies if $\ind_V$ is log transverse to $\mathcal{A}$. We then have the formula
\begin{equation}
\label{eq; csmhyperplane}
c_*(\ind_{\mathbb{P}^n\setminus V\cup \mathcal{A}}) = c(T\mathbb{P}^n(- \log \mathcal{A}))c(L)^{-1}\cap \Big([\mathbb{P}^n]-s\left(\mathcal{J}_V(\log \mathcal{A}), \mathbb{P}^n\right)^\vee\otimes L \Big).
\end{equation}
where $\mathcal{L}\cong\mathscr{O}(V)$. In \S \ref{sec; centralfree} we explained that $H^0(\mathbb{C}^{n+1},\textup{Der}_{\mathbb{C}^{n+1}}(-\log \hat{\mathcal{A}}))$ is freely generated by homogeneous derivations of degree $d_i-1$, $i=0,\ldots,n$. The numbers $d_0=1,d_1,\ldots,d_n$ are usually called the exponents of $\hat{\mathcal{A}}$. By Proposition \ref{prop; chernderlog}, $T\mathbb{P}^n(- \log \mathcal{A})$ splits. Though we will not use it, we can show that the bundle of logarithmic principal parts also splits.

\begin{proposition}\label{hyparrlogjet}
Let $\mathcal{A}$ be a centrally free projective hyperplane arrangement. For any line bundle $L$ on $\mathbb{P}^n$  we have
\[
P^1_{\mathbb{P}^n}L(\log \mathcal{A}) \cong (T^*\mathbb{P}(\log \mathcal{A})\otimes L) \oplus L \/.
\]
\end{proposition}
\begin{proof}
It is enough to construct a logarithmic connection $\nabla: \mathcal{L} \to \Omega^1_{\mathbb{P}^n}(\log \mathcal{A})\otimes \mathcal{L}$. Let $l_1\ldots l_k$ be a defining equation of the affine central arrangement $\hat{\mathcal{A}}$. Let $l$ be the degree of $L$, then any local section of $L$ can be written as a homogeneous degree $l$ rational function $g\in\mathbb{C}(x_0,\ldots,x_n)$. Let $\chi=x_0\partial_0 + \ldots + x_n\partial_n$ be the standard Euler derivation. One can check immediately that $dg-\frac{lg}{k}\cdot\frac{d(l_1\ldots l_k)}{l_1\ldots l_k}$ is a degree $l$ rational form (the differentials $dx_0,\ldots,dx_n$ have degree $1$) and $\chi(dg-\frac{lg}{k}\cdot\frac{d(l_1\ldots l_k)}{l_1\ldots l_k})=0$. Therefore $dg-\frac{lg}{k}\cdot\frac{d(l_1\ldots l_k)}{l_1\ldots l_k}$ is a well-defined local section of $T^*\mathbb{P}^n(\log \mathcal{A})\otimes L$. Define the connection $\nabla$ by
\begin{equation*}
\nabla g= dg-\frac{lg}{k}\cdot\frac{d(l_1\ldots l_k)}{l_1\ldots l_k}.
\end{equation*} 
For any degree $0$ homogeneous function $h\in \mathbb{C}(x_0,\ldots,x_n)$, we can easily check that $\nabla(hg)=h\nabla g+gdh$. Therefore $\nabla$ defines a logarithmic connection.
\end{proof}

As an application we now reformulate a recent result \cite[Theorem 1.1]{RW24} in the centrally free setting.
Let $\mathcal{A}=V(l_0\cdots l_m)$ be a  centrally free projective hyperplane arrangement 
and $\chi_{\mathcal{A}}(t):=\chi_{\tilde{\mathcal{A}}}(t)$  be the characteristic polynomial  of its affine cone $\tilde{\mathcal{A}}$. In  \cite[Theorem 1.2, Theorem 1.3]{MR3047491} Aluffi proved that, writing $c(T\mathbb{P}^n(-\log \mathcal{A}))=\sum_{k=0}^n c_k\cdot H^k$ we have
\[
t\cdot \sum_{k=0}^n c_{n-k}\cdot t^k =  \chi_{\mathcal{A}}(t+1)  
\]
 
We can   view $\mathcal{H}_0:=V(l_0)$ as the hyperplane at infinity and take the affine complement $U_0:=\{l_0\neq 0\}\cong \mathbb{C}^n$.
Let $A:=V(l_1\cdots l_m)\cap U_0$  be the restricted affine hyperplane arrangement in $U_0$ and $\chi_{A}(t)=\sum_{k=0}^n a_k \cdot t^k$ be its characteristic polynomial. 
Then $\mathcal{A}$ is the coning of $A$ and it's   known that $P_{\mathcal{A}}(t)=P_{A}(t)\cdot (t+1)$ for the Poincar\'e polynomials. 
Recall that the Poincar\'e polynomials are related to the characteristic polynomials by 
$$P_{\mathcal{A}}(t)=(-t)^{n+1}\chi_{\mathcal{A}}(-t^{-1})\/; \quad P_{A}(t)=(-t)^{n}\chi_{A}(-t^{-1})\/.$$  
Then we have 
\[
\sum_{k=0}^n c_{k}\cdot t^{k}=t^n\cdot \chi_{A}(1+\frac{1}{t})= \sum_{k=0}^n a_k \cdot (1+t)^kt^{n-k}  \/.
\]
Given a degree $d$ hypersurface $V\subset \mathbb{P}^n$  such that $\ind_V$ is log transverse to $\mathcal{A}$. 
We define  integer $b_k(V, \log \mathcal{A})$ as the  coefficient of $x^k$ in polynomial $L_{V,k}(x)$, where 
\[
L_{V,k}(H)\cap [\mathbb{P}^n]:=(1+H)^k\cdot (1+dH)^{-1} \cap \Big([\mathbb{P}^n]-s\left(\mathcal{J}_V(\log \mathcal{A}), \mathbb{P}^n\right)^\vee\otimes \mathscr{O}_{\mathbb{P}^n}(d) \Big) \/.
\] 
Since $\chi(\mathbb{P}^n\setminus V\cup \mathcal{A})$ is the coefficient of $[\mathbb{P}^0]$ in $c_*(\ind_{\mathbb{P}^n\setminus V\cup \mathcal{A}})$, substituting $L =\mathscr{O}_{\mathbb{P}^n}(d)$ in \eqref{eq; csmhyperplane} and taking degree we have
\begin{corollary}
\label{coro; akbk}
Let $V^\circ:=V\cap U_0$ be the affine hypersurface, then we have:
$$\chi(\mathbb{P}^n\setminus V \cup \mathcal{A})=\chi(\mathbb{C}^n \setminus V^\circ \cup A) =\sum_{k=0}^n a_k\cdot b_k(V,\log \mathcal{A}) \/.$$ 
\end{corollary}
 
In \cite{RW24} the authors considered  an arbitrary affine hyperplane arrangement $A\subset \mathbb{C}^n$ which is not necessarily free, and $r$ generic polynomials $g_1, \cdots ,g_r$ on $\mathbb{C}^n$ of degrees $d_1, \cdots ,d_r$. 
They defined integers $b_k[d_1, \cdots ,d_r]$  and proved  in \cite[Theorem 1.1]{RW24} that:
$$\chi(\mathbb{C}^n\setminus A\cup V(g_1\cdots g_r))=\sum_{k=0}^n a_k\cdot b_k[d_1, \cdots ,d_r]\/.$$ 
Here $\chi_{A}(t)=\sum_{k=0}^n a_k \cdot t^k$ is the characteristic polynomial of $A$.

Embed $\mathbb{C}^n$ into $\mathbb{P}^n$,  the coning of $A$
is  $\mathcal{A}:=\bar{A}\cup \mathcal{H}_\infty$, the union   of the closure of $A$ with the hyperplane at infinity $\mathcal{H}_\infty$.  
Let $\bar{g}_i$ be the homogenization of $g_i$, we have  
$$V(\bar{g}_1\cdots \bar{g}_r)\cap \mathbb{C}^n=V(g_1\cdots g_r) \text{ and } \mathbb{C}^n\setminus A\cup V(g_1\cdots g_r)=\mathbb{P}^n\setminus \mathcal{A}\cup V(\bar{g}_1\cdots \bar{g}_r)\/.$$ 
Thus when the coning  $\mathcal{A}$ is  centrally free,  the formula in Corollary~\ref{coro; akbk} has the same format as in \cite[Theorem 1.1]{RW24}.

The two formulas agree when $A$  is SNC. 
Let $\textup{Sing}(V)$ be the singular subscheme of $V:=V(\bar{g}_1\cdots \bar{g}_r)$.
The coning $\mathcal{A}$ is also SNC and hence  $s\left(\mathcal{J}_V(\log \mathcal{A}), \mathbb{P}^n\right)=s(\textup{Sing}(V), \mathbb{P}^n)$ (see Theorem~\ref{theo; Segreimpliesintegral}).  Then by \cite[Theorem 2.1]{Aluffi03} and \cite[Proposition 2.6]{Aluffi13} we have 
\begin{align*}
H^{n-k}L_{V,k}(H)\cap [\mathbb{P}^n]=& \frac{H^{n-k}(1+H)^k}{(1+dH)}\cap \Big([\mathbb{P}^n]-s(\textup{Sing}(V), \mathbb{P}^n)^\vee\otimes \mathscr{O}_{\mathbb{P}^n}(d) \Big)   \\
=&  \frac{H^{n-k}}{(1+H)^{n-k}}\cdot \frac{(1+H)^{n}}{(1+dH)} \cap \Big([\mathbb{P}^n]-s(\textup{Sing}(V), \mathbb{P}^n)^\vee\otimes \mathscr{O}_{\mathbb{P}^n}(d) \Big) \\
=& \frac{H^{n-k}}{(1+H)^{n-k}} \cap c_*(\ind_{\mathbb{P}^n\setminus V\cup \mathcal{H}})  
=  c_*\left(\ind_{\mathbb{P}^k\setminus V\cup \mathcal{H}}\right) \in H_*(\mathbb{P}^n) 
\end{align*}
for a generic linear subspace $\mathbb{P}^k$ and a generic hyperplane $\mathcal{H}$. 
Then $b_k(V, \log \mathcal{A})$  equals the coefficient of $[\mathbb{P}^0]$ in $c_*(\ind_{\mathbb{P}^k\setminus V\cup \mathcal{H}})$  and hence equals $\chi(\mathbb{P}^k\setminus V\cup \mathcal{H})$. From \cite[Remark 2.9]{RW24} we see that
\[
b_k(V, \log \mathcal{A})=b_k[d_1, \cdots d_r] \text{ for } k=0, \cdots ,n  \/.
\]
So Corollary~\ref{coro; akbk} coincides with \cite[Theorem 1.1]{RW24} when $A$ is simple normal crossing.

\begin{question}
Let $\mathcal{A}$ be any centrally free projective hyperplane arrangement and let $V:=V(\bar{g}_1\cdots \bar{g}_r)$ be a generic hypersurface as in \cite[Theorem 1.1]{RW24}. Do we still have $b_k(V, \log \mathcal{A})=b_k[d_1, \cdots d_r]$? Since $b_k[d_1, \cdots d_r]$ is independent of $\mathcal{A}$, is $b_k(V, \log \mathcal{A})$ also independent of $\mathcal{A}$ when $\mathcal{A}$ is centrally free?
\end{question}

\begin{remark}
In fact,  Theorem 1.1 and  Theorem 1.3 in \cite{RW24} can be directly obtained from the transverse product formula   \cite[Theorem 1.2]{MR3706222}. Let $g_1, \cdots ,g_r$ be $r$ generic polynomials on $\mathbb{C}^n$ and $\bar{g}_1, \cdots, \bar{g}_r$ be their homogenizations. 
The key observation is that 
\[
 b_j[d_1, \cdots ,d_r] = \text{ the coefficient of } [\mathbb{P}^0] \text{ in } 
 \frac{H^{n-j}}{(1+H)^{n-j}}\cap c_*\left(\ind_{\mathbb{P}^n\setminus V(\bar{g}_1\cdots \bar{g}_r)} \right) \/.
\]
Given $f_1, \cdots ,f_k$ as in \cite[Theorem 1.3]{RW24}, following \cite{Aluffi03} we may write
\[
c_*\left(\ind_{\mathbb{C}^n\setminus V(f_1\cdots f_k)}  \right)=c_*\left(\ind_{\mathbb{P}^n\setminus V(\bar{f}_1\cdots \bar{f}_k)\cup \mathcal{H}_\infty}  \right)= (1+H)\sum_{i=0}^n \mu_{n-i} \cdot H^{i}(1+H)^{n -i} \cap [\mathbb{P}^n]\/.
\]
Since  $g_1, \cdots ,g_r$ are generic,  \cite[Theorem 1.2]{MR3706222}  applies:
\begin{align*}
& c_*\left(\ind_{\mathbb{P}^n\setminus \mathcal{H}_\infty \cup V(\bar{f}_1\cdots \bar{f}_r)\cup V(\bar{g}_1\cdots \bar{g}_r)} \right)  
= c_*\left(\ind_{\mathbb{P}^n\setminus \mathcal{H}_\infty \cup V(\bar{f}_1\cdots \bar{f}_r)}\cdot \ind_{\mathbb{P}^n\setminus V(\bar{g}_1\cdots \bar{g}_r)} \right) \\
=&  \frac{1}{(1+H)^{n+1}}\cap c_*\left(\ind_{\mathbb{P}^n\setminus \mathcal{H}_\infty \cup V(\bar{f}_1\cdots \bar{f}_r)} \right) 
\cdot c_*\left(\ind_{\mathbb{P}^n\setminus V(\bar{g}_1\cdots \bar{g}_r)}\right)\\
=& \left( \sum_{i=0}^n \mu_{n-i} \cdot H^{i}(1+H)^{n+1-i} \right)\cap  \frac{c_*\left(\ind_{\mathbb{P}^n\setminus V(\bar{g}_1\cdots \bar{g}_r)}\right)}{(1+H)^{n+1}} \\
=& \left( \sum_{i=0}^n \mu_{n-i} \cdot  \frac{ H^{i}}{(1+H)^{i}} \right) \cap c_*\left(\ind_{\mathbb{P}^n\setminus V(\bar{g}_1\cdots \bar{g}_r)}\right) \/.
\end{align*}
Tracking the coefficient of $[\mathbb{P}^0]$ we then have
\[
\chi\left(\mathbb{C}^n\setminus  V(f_1\cdots f_r)\cup V(g_1\cdots g_r)\right)=
\sum_{i=0}^n \mu_i\cdot b_i[d_1, \cdots d_r] \/.
\]
These $\{\mu_i\}$ are the coefficients of the characteristic polynomial in  Theorem 1.1, and are the Milnor numbers in Theorem 1.3 of \cite{RW24}.
\end{remark}

\subsection{A Comparison of Segre Classes for SNC Divisors}

In this subsection we consider SNC divisors $\mathcal{A}_k:=V(x_0\cdots x_k)\subset \mathbb{P}^n$ for $k\leq n$.  We fix a  basis for   $\textup{Der}_{\mathbb{C}^{n+1}}(-\log \tilde{\mathcal{A}}_k)$:
$$\mathcal{B}_k:=\{x_0\partial_{x_0}, \cdots ,x_k\partial_{x_k}, \partial_{x_{k+1}},\cdots ,\partial_{x_n}\} \/.$$  

Let $V=V(g)\subset \mathbb{P}^n$ be a degree $d$  hypersurface such that  $\ind_V$ is log transverse to $\mathcal{A}_k$. 
This can be achieved, for example,  if  $V$ admits a Whitney stratification 
$\{V_\alpha\}$ such that no stratum $V_\alpha$ is contained in  $\mathcal{A}_k$. 
In particular, a general $PGL_n$-translate of $V$ is log transverse to $\mathcal{A}_k$ for any $k$. 
 We define the following ideal sheaves:
$$\mathcal{J}_kV:=\mathcal{J}_V(\log\mathcal{A}_k)=
(x_0g_{x_0},  \cdots ,x_kg_{x_k}, g_{x_{k+1}}, \cdots ,g_{x_n} )\subset \mathscr{O}_{\mathbb{P}^n} \/.
$$
Clearly we have inclusions $\mathcal{J}_{k+1}V \subset \mathcal{J}_kV\subset \mathcal{J}_V=(g_{x_0},  \cdots , g_{x_n})$  for any $k$. 

Assume  $\ind_V$ is log transverse to $\mathcal{A}_k$, combining Theorem~\ref{theo; csmdoubleres} and Example~\ref{exam; logcothyparr} we have
\begin{equation}
\label{VcupAk}
c_{sm}(\ind_{\mathbb{P}^n\setminus(V\cup \mathcal{A}_k)}) = \frac{(1+H)^{n-k}}{(1+dH)}\cap \Big([\mathbb{P}^n]-s\left(\mathcal{J}_kV, \mathbb{P}^n\right)^\vee\otimes \mathscr{O}_{\mathbb{P}^n}(d) \Big)  
\end{equation}

\label{sec; comparisonofSegreclass}
The following  comparison on the Segre classes is motivated from \cite{Aluffi13} and \cite[\S 4]{FMS}. 
In \cite[Proposition 2.6]{Aluffi13} Aluffi proved that, 
for a locally closed subset $Y\subset \mathbb{P}^n$ and a generic hyperplane $\mathcal{H}$, in  $H_*(\mathbb{P}^n)$  we have
\[
(1+H)\cap c_*(\ind_{Y\cap \mathcal{H}}) = H\cap c_*(\ind_Y)\quad   \text{  and  }\quad  (1+H)\cap c_*(\ind_{Y\setminus \mathcal{H}}) =  c_*(\ind_Y)   \/.
\]  
 
Set  $Y=\mathbb{P}^n\setminus V$.  
For generic hyperplanes $H_0, \cdots , H_k$, up to a linear coordinate change we may assume  $\cup_{i=0}^k H_k=\mathcal{A}_k$ and    $\ind_V$ is log transverse to $\mathcal{A}_k$. Then we get
\[
(1+H)^{k+1} \cap c_*(\ind_{\mathbb{P}^n\setminus V\cup \mathcal{A}_k})=c_*(\ind_{\mathbb{P}^n\setminus V}) \/.
\]

Compare this formula to \eqref{VcupAk}   we have
\[
\frac{(1+H)^{n+1}}{(1+dH)}\cap \Big([\mathbb{P}^n]-s\left(\mathcal{J}_kV, \mathbb{P}^n\right)^\vee\otimes \mathscr{O}_{\mathbb{P}^n}(d) \Big) =c_*(\ind_{\mathbb{P}^n\setminus V})
\] 
As $\mathcal{J}_V$ is the ideal sheaf of $\textup{Sing}(V)$, from \cite[Theorem 2.4]{Aluffi13}  we have
\[
\frac{(1+H)^{n+1}}{(1+dH)} \cap \Big([\mathbb{P}^n]-s\left(\mathcal{J}_kV, \mathbb{P}^n\right)^\vee\otimes \mathscr{O}_{\mathbb{P}^n}(d) \Big) =
\frac{(1+H)^{n+1}}{(1+dH)} \cap \Big([\mathbb{P}^n]-s\left(\mathcal{J}_V, \mathbb{P}^n\right)^\vee\otimes \mathscr{O}_{\mathbb{P}^n}(d) \Big)  \/.
\]
  
Thus for any $k\leq n$  we have:
\begin{equation}
\label{eq; segreJkJ}
s\left(\mathcal{J}_kV, \mathbb{P}^n\right) =
s\left(\mathcal{J}_V, \mathbb{P}^n\right)\in H_*(\mathbb{P}^n) \/.
\end{equation}
In an earlier draft of \cite{FMS} the authors discussed this equality under a normal crossing intersection hypothesis by proving that in such case  we have $\mathcal{J}_kV=\mathcal{J}_V$. 
We now  provide a general explanation of this  equality (see Theorem~\ref{theo; Segreimpliesintegral}):
there   is a sequence of finite morphisms  of blowups which are  compatible with  exceptional divisors  
\[
Bl_{\mathcal{J}_V}\mathbb{P}^n\to Bl_{\mathcal{J}_0V}\mathbb{P}^n \to \ldots \to Bl_{\mathcal{J}_{n-1}V}\mathbb{P}^n \to Bl_{\mathcal{J}_nV}\mathbb{P}^n   \/.
\]    
In particular all the Segre classes $s\left(\mathcal{J}_kV, \mathbb{P}^n\right)$   equal  the  pushforward of   $s(\mathcal{J}_V, \mathbb{P}^n)$.

Now we explain. To start we recall the Segre cycles introduced in \cite{G-G99} by Gaffney and Gassler.
Given an ideal $I\subset \mathscr{O}_{\mathbb{C}^{n+1}, 0}$ with $K$ generators, we construct
the blowup $Bl_{I}\mathbb{C}^{n+1}$  with exceptional divisor $\mathcal{E}$. Let 
$\begin{tikzcd}
\mathbb{C}^{n+1} & Bl_{I}\mathbb{C}^{n+1} \arrow[l, "b"']  \arrow[r, "q"] & \mathbb{P}^{K-1}
\end{tikzcd}$
be the  projections.  
The $i$-th Segre cycle  $\Lambda_i(I, \mathbb{C}^{n+1})$  is defined by first intersecting $\mathcal{E}$ with the  pullback of a generic  $(i-1)$-codimensional linear subspace through $q$, then pushing forward the intersection cycle through $b$. This is a cycle of $\mathbb{C}^{n+1}$ whose support contains $0$, and   the Segre multiplicities of $I$ is defined by $e_k(I, \mathbb{C}^{n+1}):=\textup{multi}_0(\Lambda_i(I, \mathbb{C}^{n+1}))\/.$  
 
When   $I$ is homogeneous  we take the blowup $\bar{b}\colon Bl_I\mathbb{P}^n\to \mathbb{P}^n$ with exceptional $E$. If $I$ is generated by homogeneous elements of the same degree,  the Segre multiplicities  of $I$
then equal 
the coefficients of  the Segre class  $s(I, \mathbb{P}^n)$: 
\[
s(I, \mathbb{P}^n)=\bar{b}_*s(E,Bl_I\mathbb{P}^n)=\sum_{i=1}^n e_i(I, \mathbb{C}^{n+1})\cdot H^i\cap [\mathbb{P}^n] \/.
\]

Let $J_V=(g_{x_0}, \cdots ,g_{x_n})$ and $J_kV=(x_0g_{x_0},  \cdots ,x_kg_{x_k}, g_{x_{k+1}}, \cdots ,g_{x_n})$ be the corresponding homogeneous ideals in $\mathscr{O}_{\mathbb{C}^{n+1}}$. 
Both $J_V$ and $J_nV$ are  generated by homogeneous elements of the same degree, and $e_i(J_nV, \mathbb{C}^{n+1})=e_i(J_V, \mathbb{C}^{n+1})$ for   $i=1, \cdots ,n$ by 
Equation $(\ref{eq; segreJkJ})$. 
We now show that these numerical constraints force  the ideals to have the same integral closure. 
For details about  integral dependence we refer to \cite{G-G99} and \cite{LJ-T08}. 
\begin{theorem}
\label{theo; Segreimpliesintegral}
Assume   $\ind_V$ is log transverse to $\mathcal{A}=V(x_0\cdots x_n)$.  
Then for each $k$  the  ideal  sheaves $\mathcal{J}_kV\subset \mathcal{J}_V$  have the same integral closure, i.e., $\overline{\mathcal{J}_kV} =\overline{\mathcal{J}_V}$ in $\mathscr{O}_{\mathbb{P}^n}$.  Equivalently this says that
the inclusion $\mathcal{J}_kV\subset \mathcal{J}_V$  induces a finite  morphism  $Bl_{\mathcal{J}_V} \mathbb{P}^n\to Bl_{\mathcal{J}_kV} \mathbb{P}^n$.
\end{theorem}
\begin{proof} 
It suffices to prove the statement for $k=n$,    since for each $k$ we have  
$$\overline{\mathcal{J}_nV} \subset \overline{\mathcal{J}_kV} \subset \overline{\mathcal{J}_V} =\overline{\mathcal{J}_nV}\/.$$

By \cite{LJ-T08} it is equivalent to prove that the inclusion $\mathcal{J}_nV\subset \mathcal{J}_V$ induces a finite  morphism $Bl_{\mathcal{J}_V} \mathbb{P}^n\to Bl_{\mathcal{J}_nV} \mathbb{P}^n$.
The inclusion of affine ideals $J_nV\subset J_V$   induces an inclusion of Rees algebras 
and therefore a $\mathbb{C}^*$-equivariant morphism of affine spectrums:
$$\tilde{\pi} \colon \textup{Spec}\left( \textup{Rees}_{\mathscr{O}_{\mathbb{C}^{n+1}}}(J_V)\right) \to \textup{Spec}\left( \textup{Rees}_{\mathscr{O}_{\mathbb{C}^{n+1}}}(J_nV)\right) \/.
$$

By previous discussion Equality \eqref{eq; segreJkJ} implies that
\[
e_i(J_nV, \mathbb{C}^{n+1})=e_i(J_V, \mathbb{C}^{n+1}) \text{ for } i=1, \cdots ,n \/.
\]  
Using the arguments in \cite[\S 4]{G-G99} one can show that $\tilde{\pi}$ has finite fibers away from the fibers over $0\in \mathbb{C}^{n+1}$. Then the morphism 
$\pi \colon \textup{Spec}\left( \textup{Rees}_{\mathscr{O}_{\mathbb{P}^n}}(\mathcal{J}_V)\right) \to \textup{Spec}\left( \textup{Rees}_{\mathscr{O}_{\mathbb{P}^n}}(\mathcal{J}_nV)\right)$  
has finite fibers. Passing to the projectivizations the following morphism is well defined:
\[
\pi \colon Bl_{\mathcal{J}_V} \mathbb{P}^n=\textup{Proj}\left( \textup{Rees}_{\mathscr{O}_{\mathbb{P}^n}}(\mathcal{J}_V)\right) \longrightarrow \textup{Proj}\left( \textup{Rees}_{\mathscr{O}_{\mathbb{P}^n}}(\mathcal{J}_nV)\right)=Bl_{\mathcal{J}_nV} \mathbb{P}^n \/.
\]
This  is a proper morphism with finite fibers, so it is a finite morphism. 
\end{proof} 
 
\begin{remark}
However $J_kV\subset J_V$ may not be a reduction as ideals of $\mathscr{O}_{\mathbb{C}^{n+1}}$. The generalised Rees' Theorem~\cite[Corollary 4.9]{G-G99} says that
$\overline{J_kV} =\overline{J_V}$ as ideals of $\mathscr{O}_{\mathbb{C}^{n+1}}$ if and only if $e_i(J_kV, \mathbb{C}^{n+1})=e_i(J_V, \mathbb{C}^{n+1})$ for any $i=1, \cdots ,n+1$. Here the $(n+1)$-th Segre numbers
$e_{n+1}(J_kV, \mathbb{C}^{n+1})$ and $e_{n+1}(J_V, \mathbb{C}^{n+1})$ are out of our control. 

For example, consider $g=x^2+y^2+z^2$, then as ideals in $\mathscr{O}_{\mathbb{C}^3}$  we have $J_2V=(x^2, y^2, z^2)$ and $J_V=(x,y,z)$. Then Rees' theorem says that as $\mathscr{O}_{\mathbb{C}^{3}}$-ideals $\overline{J_2V}\neq \overline{J_V}$ since $e(J_2V)\neq e(J_V)$. However, as ideals in $\mathscr{O}_{\mathbb{P}^2}$ we have $V(\mathcal{J}_2V)=V(\mathcal{J}_V)=\emptyset$.
\end{remark}

\bibliographystyle{plain}
\bibliography{liao}
\end{document}